\documentclass[10pt]{amsart}

\pdfoutput=1

\usepackage{graphicx, arcs, amsmath, amsthm, amssymb, enumerate, esint, mathtools}
\usepackage{amsfonts,mathrsfs,color,multirow}
\usepackage{pgf,tikz}
\usetikzlibrary{arrows}
\usepackage[toc,page]{appendix}

\graphicspath{ {./images/} }
\numberwithin{equation}{section}

\newcommand{\limit}[3]{\displaystyle \lim_{#1 \to #2} #3}

\newcommand{\R}[1]{\mathbb{R}^{#1}}
\newcommand{\Z}{\mathbb{Z}}
\newcommand{\Q}{\mathbb{Q}}

\newcommand{\pprime}{{\prime\prime}}

\newcommand{\fpd}[2]{\frac{\partial #1}{\partial #2}}

\newcommand{\bs}[1]{\boldsymbol{#1}}
\newcommand{\twobytwo}[4]{\left(\begin{array}{cc}
#1 & #2  \\
#3 & #4 \end{array} \right)}
\newcommand{\ip}[2]{\left<  #1, #2 \right> }

\theoremstyle{plain}
\newtheorem{theorem}{Theorem}

\newtheorem{prop}{Proposition}

\newtheorem{cor}{Corollary} 
\newtheorem{defn}{Definition}
\newtheorem{remark}{Remark}

\allowdisplaybreaks

\title[On the Dynamics of Inverse Magnetic Billiards]{On the Dynamics of Inverse Magnetic Billiards}
\author{Sean Gasiorek}
\address{School of Mathematics and Statistics, Carslaw Building F07, University of Sydney, NSW 2011 Australia}
\email{sean.gasiorek@sydney.edu.au}

\begin{document}

\begin{abstract}
Consider a strictly convex set $\Omega$ in the plane, and a homogeneous, stationary magnetic field orthogonal to the plane whose strength is $B$ on the complement of $\Omega$ and $0$ inside $\Omega$. The trajectories of a charged particle in this setting are straight lines concatenated with circular arcs of Larmor radius $\mu$. We examine the dynamics of such a particle and call this \emph{inverse magnetic billiards}. Comparisons are made to standard Birkhoff billiards and magnetic billiards, as some theorems regarding inverse magnetic billiards are consistent with each of these billiard variants while others are not. 
\end{abstract}

\maketitle



\section{Introduction}\label{intro}

Consider the classical motion of a particle of mass $m$ and charge $e$ in the plane. Let $\Omega \subset \R{2} $ denote a connected, strictly convex domain, and define a constant, homogeneous, stationary magnetic field orthogonal to the plane which has strength $B$ on $\R{2} \setminus \Omega$ and 0 on $\Omega$. As such, the equations of motion for the particle of position $q$ and velocity $v$ are as follows: \newline
$$ \begin{cases} \dot{q} &= v \\ \dot{v} &= B_\Omega(q) \mathbb{J}v \end{cases} \hspace{0.3cm} \text{with} \hspace{0.3cm} \mathbb{J} := \twobytwo{0}{-1}{1}{0}, \hspace{0.3cm} B_\Omega(q) := \begin{cases} 
      0 & q \in \Omega \\
      B & q \in \R{2} \setminus \Omega  
   \end{cases}.  $$ \\
The solution to this initial value problem are continuous curves which are circular arcs outside $\Omega$ and straight lines inside $\Omega$. The circular arcs will have Larmor radius $\mu=\frac{m|v|}{|eB|}$, and speed $|\dot{q}|$ and energy $E$ are constants of motion. 
Without loss of generality we assume $e<0$ and $B >0$ so that the motion along the circular arcs
will be traversed in the counterclockwise direction. 

Following the construction in \cite{BK}, suppose the boundary $\partial \Omega$ is $C^k$ with $k \geq 3$ and total length $|\partial \Omega | =L$. The boundary $\partial \Omega = Image(\Gamma(s))$ will be parametrized by arc length, $s$:
$$ \Gamma(s) = (X(s), Y(s)), \hspace{0.5cm} ds^2 = dX^2 + dY^2, \hspace{0.5cm} s \in \R{}/L\Z.$$ 
The unit tangent and unit normal vectors and curvature are given by 
\begin{align*}
\bs{t}(s) &= (X^\prime (s), Y^\prime(s)) = (\cos(\tau(s)), \sin(\tau(s))), \\
\bs{n}(s) &= (-Y^\prime(s),X^\prime(s)),\\
\kappa(s) &= \frac{d\tau}{ds} = X^\prime(s)Y^\pprime(s) - X^\pprime(s) Y^\prime(s) = \frac{1}{\rho(s)},
\end{align*}
so that $\tau(s)$ is the polar angle between the positive $x$-axis and $\bs{t}(s)$, and $\rho(s)$ is the radius of curvature. Because $\Omega$ is strictly convex the curvature of the boundary is strictly positive and $\rho(s)$ is bounded by positive constants, $0 < \rho_{min} \leq \rho(s) \leq \rho_{max} < \infty$ for all $s$. Following the lead of \cite{RB}, we will explore the dynamics of our system  in terms of the relative sizes of the Larmor radius $\mu$ and the maximum and minimum radii of curvature of $\partial \Omega$. We will refer to these possibilities $$\mu < \rho_{min}, \;\;\;\;\;\; \rho_{min} < \mu < \rho_{max}, \;\;\;\;\;\; \rho_{max} < \mu$$ as \emph{curvature regimes}. The billiard flow is hence given by the Lagrangian $$\mathcal{L}(q,\dot{q}) = \frac{1}{2}m |\dot{q}|^2 + e\ip{\dot{q}}{\mathbb{A}(q)}, \;\;\;\; \mathbb{A}(q) =\frac{1}{2}(-yB_\Omega(q), xB_\Omega(q)) =  \frac{1}{2} B_\Omega(q)\mathbb{J}q $$ where $\ip{\cdot}{\cdot}$ is the standard Euclidean inner product. We call this dynamical system \emph{inverse magnetic billiards}, following the naming by \cite{VTCP}. \\

%
%
%
 Electron dynamics in piecewise-constant magnetic fields are studied in \cite{CP}, \cite{KPC}, \cite{KROC}, \cite{No}, \cite{SIKL}, \cite{SG}, and \cite{VTCP}. Classical, semiclassical, and quantum approaches to this system are each addressed to a degree -- occasionally in compact subsets and sometimes in unbounded regions -- but none are in-depth mathematically to the extent of \cite{BK} with respect to magnetic billiards, for example. \\


This paper is strongly influenced by the work of Berglund and Kunz in \cite{BK}, and is organized as follows. Section 2 gives a thorough description of the billiard flow and describes its motion through a return map $T$. An exact expression is given for the Jacobian $DT$. The map $T$ is sometimes a twist map and admits a generating function $G$, which is given explicitly in section 3. In section 4 we address the existence of periodic orbits using $G$. Some calculations are made in section 5 that are specific to the ellipse. Section 6 details the existence and nonexistence of caustics using approaches similar to Mather, Berglund, and Kunz. 

\section{Constructing the return map}\label{RetMapConst}

As the particle moves, it successively leaves and re-enters $\Omega$ at the points $P_0, P_1, \newline P_2, P_3 \ldots \in \partial \Omega$. Index these points so that points with even index $P_0, P_2, P_4, \ldots$ are re-entry points and points $P_1, P_3, P_5, \ldots$ of odd index are exit points.   Express the  oriented line segment  $P_0 P_1$ joining each  entry point to its  successive exit point as a vector   $\ell_1 \vec v_0 = P_1 - P_0 $ where $\vec v_0$ is the unit vector representing the direction of motion of the particle while it travels inside $\Omega$ from $P_0$ to $P_1$ and where $\ell_1 = |P_0P_1|$ is the chord distance it travels.

The entire dynamics is summarized by the map $T: M \to M$ which takes $(P_0, v_0)$ to $(P_2, v_2)$,
sending reentry point and direction  to successive re-entry point and direction. The phase space $M$ of the map $T$ consists of unit vectors $(P_i,v_i)$ whose base points $P$ are on $\partial\Omega$ with inward direction $v$. We call this map the \emph{return map} and will express it in terms of the Birkhoff coordinates used in standard billiards. Coordinatize $P_0$   by its arc length parameter $s_0$ and the vector  $v_0$ by the negative cosine of the angle $\theta_0$ between the tangent to $\Gamma$ at $P_0$ and this vector. Writing $u_i = -\cos(\theta_i)$ we call $(s_i, u_i)$ the Birkhoff coordinates of the trajectory as
it exits or re-enters $\Omega$ at $P_i$. 

The phase space $M$ can be identified with the annulus $\mathcal{P} = \R{}/L\Z \times [-1,1] \cong S^{1} \times [-1,1]$, the return map $T$ can then be written 
as a map $$T: \mathcal{P} \to \mathcal{P}, \;\;\;\; (s_{2i}, u_{2i}) \mapsto (s_{2i+2}, u_{2i+2})$$ so that $T$ is a smooth map of the closed annulus, $\mathcal{P}$. Further, the restriction $T|_{\partial\mathcal{P}} = \text{Id}_\mathcal{P}$, where the boundary $\partial \mathcal{P}$ of $\mathcal{P}$ is the usual boundary of $\mathcal{P}$, namely $\left(S^1 \times \{-1\} \right)\cup \left( S^1 \times \{1 \}\right)$. \\

 At times it may be easier to work with $T$ as a map in terms of $(s_i,\theta_i)$. In particular, we will compute Taylor expansions of $T$ in section \ref{CausticsChapter} in terms of $s$ and $\theta$. With this interpretation, we see the inverse magnetic billiard as a discrete dynamical system. \\

By construction define $\ell_i = |P_{i-1} P_i|$ for integers $i$. 
Define $\mathcal{A}_{2i,2i+2}$ to be the area between chord $P_{2i+1}P_{2i+2}$ and $\Gamma(s)$ that is also inside the Larmor circle, and let $\gamma_{2i,2i+2}$ be the circular arc of Larmor radius $\mu$ that is outside $\Omega$. Let $\mathcal{S}_{2i,2i+2}$ be the area within the circular arc $\gamma_{2i,2i+2}$ and outside $\Omega$. Define $\chi_{2i,2i+2}$ to be the angle measured counterclockwise from $\overrightarrow{P_{2i}P_{2i+1}}$ to $\overrightarrow{P_{2i+1}P_{2i+2}}$. See Figure \ref{stdpic1} for the case when $i=0$. 

\begin{figure}[tb]
\begin{center}
\includegraphics[width=0.6\textwidth]{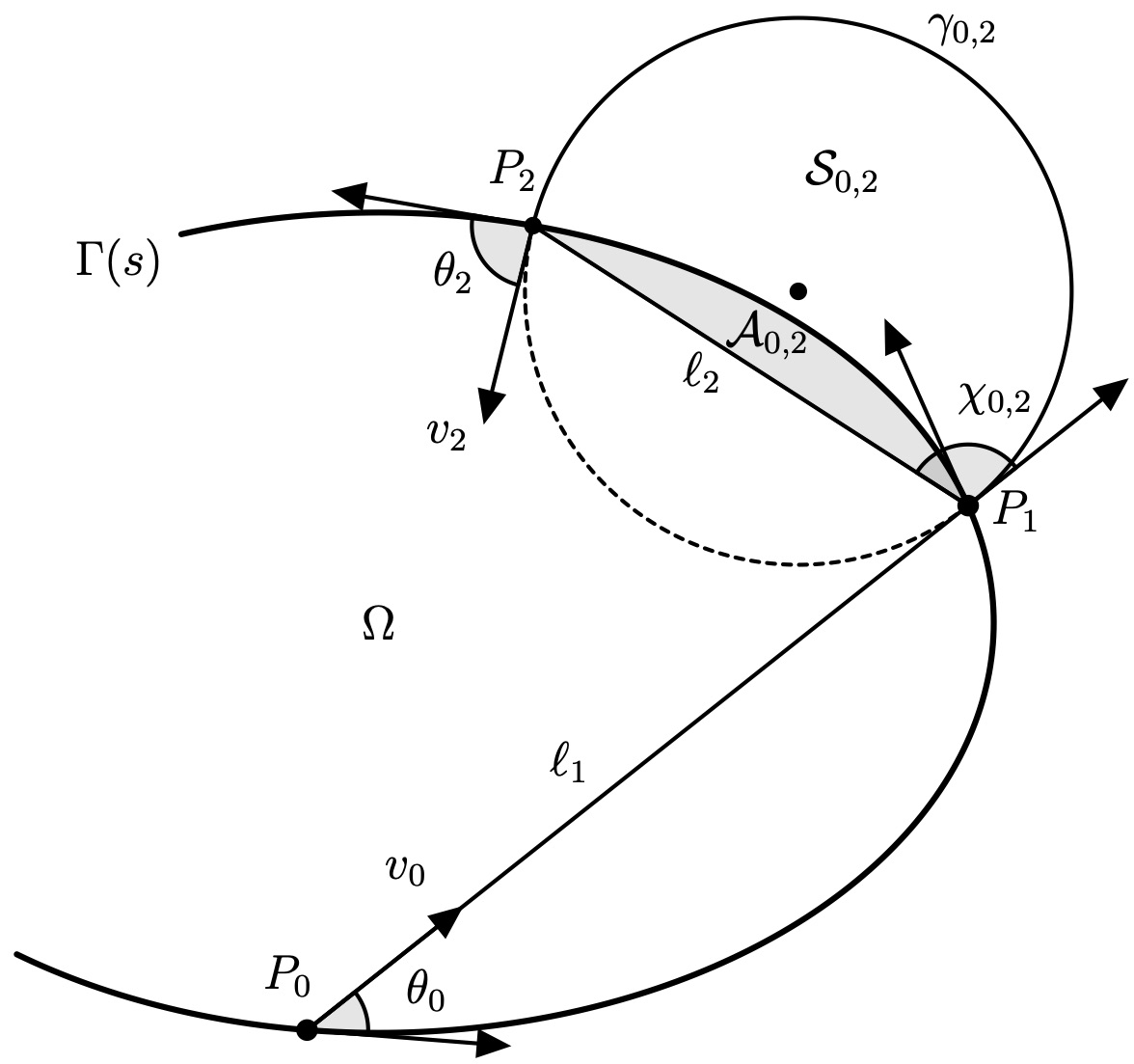} 
\end{center}
\caption{The standard picture of the return map, $T$. }
\label{stdpic1}
\end{figure} 


\begin{remark} For notational simplicity, we now omit the subscripts $2i,2i+2$, assume $i=0$, and recognize each of the described quantities below are associated to a single iteration of the return map $T$ and its realized trajectory.
\end{remark}

 Consider the magnetic arc, $\gamma$. Let the angle of such an arc be $\psi$, $\varepsilon = 2\pi-\psi$, $\delta$ is the angle between the chord $P_1P_2$ and the radius of the arc connecting each of $P_1$ and $P_2$ to the center of $\gamma$. See Figure \ref{magarcdetailsuppchi}a. From the definition of these angles and elementary geometry we find that 
$$\psi = 2\chi ~~~ \text{ and } ~~~ \sin(\chi) = \frac{\ell_2}{2\mu}.$$ It is important to note that there may be two trajectories with supplementary $\chi$ for a given chord length $\ell_2$. This is a characteristic effect of magnetic billiards. See Figure \ref{magarcdetailsuppchi}b for such an example.

\begin{figure}[ht]
\begin{tabular}{ l l } 
(a) \includegraphics[width = 0.47\textwidth]{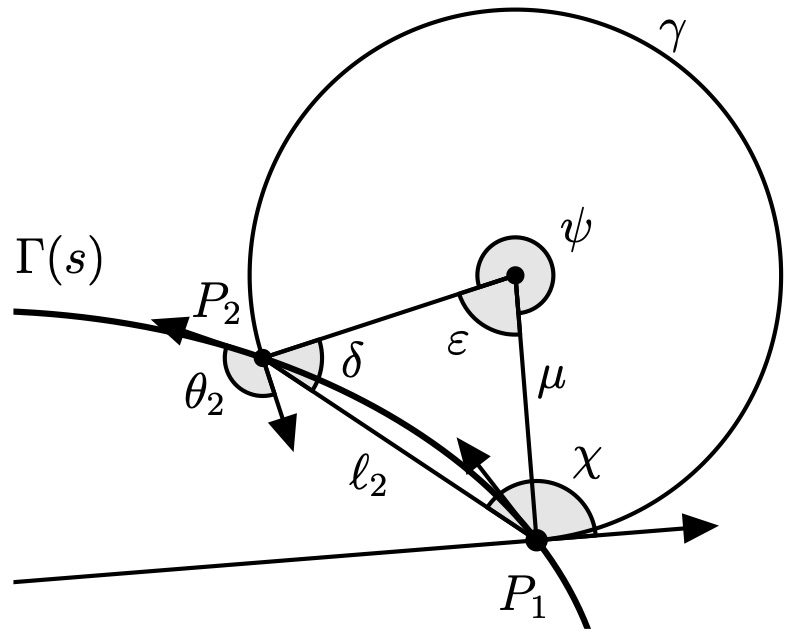} & (b) \includegraphics[width = 0.43\textwidth]{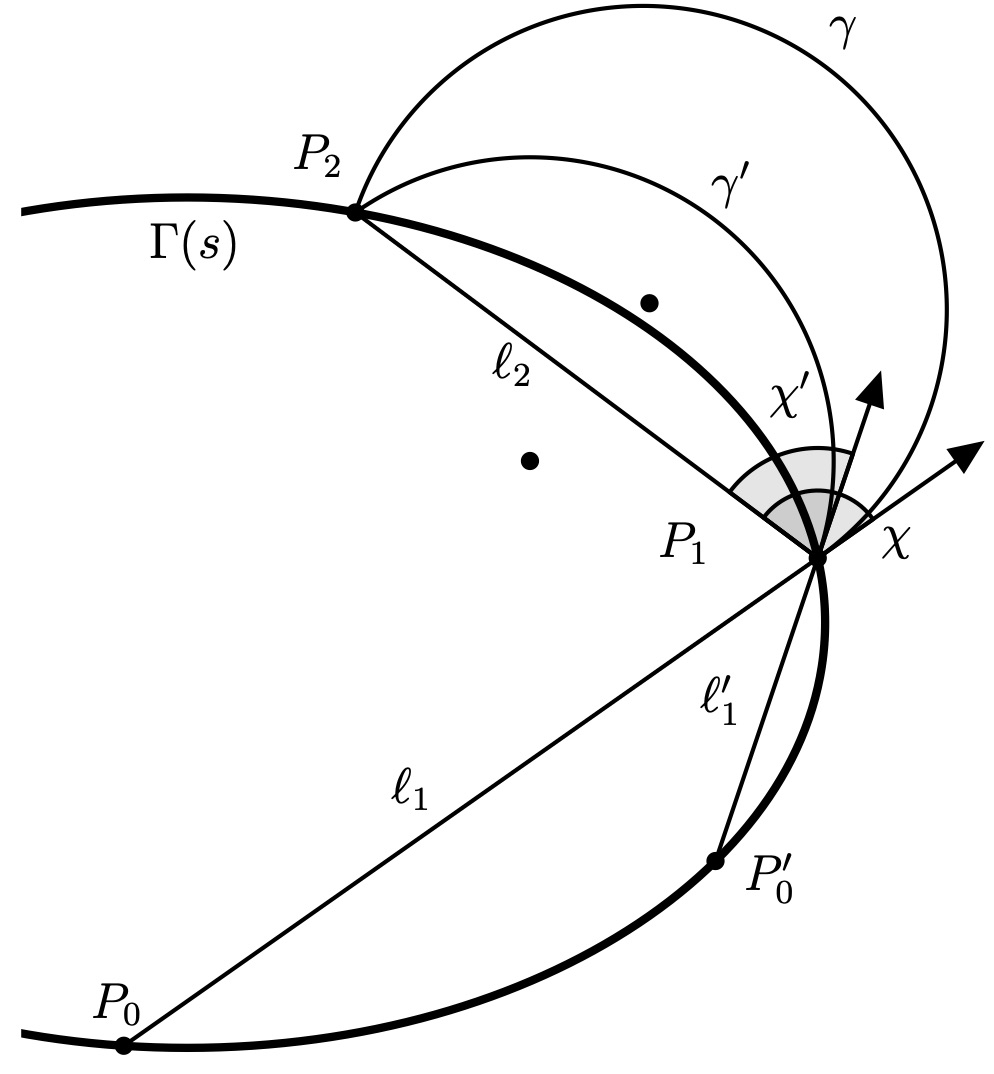} \\
\end{tabular}
\caption{(a) A magnetic arc. (b) An example of two trajectories with the same $\ell_2$ where $\chi$ and $\chi^\prime$ are supplementary. }
\label{magarcdetailsuppchi}
\end{figure}

%

We decompose $T$ into its two distinct pieces. Define the map $T_1: (s_0, u_0) \mapsto (s_1,u_1)$ as the analogue to the standard billiard map. 
The map $T_2:(s_1, u_1) \mapsto (s_2,u_2)$ is the particle moving from $P_1$ along the circular arc $\gamma$ of Larmor radius $\mu$ until intersecting $\partial \Omega$ again at $P_2$. Thus $T = T_2 \circ T_1$. 

\begin{prop}\label{JacCalc}
Given the maps $T_1$ and $T_2$, the Jacobians $DT_1 = \twobytwo{\fpd{s_1}{s_0}}{\fpd{s_1}{u_0}}{\fpd{u_1}{s_0}}{\fpd{u_1}{u_0}}$ and $DT_2  = \twobytwo{\fpd{s_2}{s_1}}{\fpd{s_2}{u_1}}{\fpd{u_2}{s_1}}{\fpd{u_2}{u_1}}$ have components

\begin{tabular}{ l l }
$\displaystyle \fpd{s_1}{s_0} = \frac{\kappa_0 \ell_1-\sin(\theta_0)}{\sin(\theta_1)}$ &  $\displaystyle \fpd{s_1}{u_0} = \frac{\ell_1}{\sin(\theta_0)\sin(\theta_1)}$ \\\\
$\displaystyle \fpd{u_1}{s_0} = \kappa_0\kappa_1\ell_1 - \kappa_1\sin(\theta_0) - \kappa_0\sin(\theta_1)$ & $\displaystyle \fpd{u_1}{u_0} = \frac{\kappa_1\ell_1 - \sin(\theta_1)}{\sin(\theta_0)}$ \\\\
$\displaystyle \fpd{s_2}{s_1} = \frac{\sin(2\chi -\theta_1)-\kappa_1\ell_2\cos(\chi)}{\sin(\theta_2)}$ & $\displaystyle \fpd{s_2}{u_1} = \frac{\ell_2\cos(\chi)}{\sin(\theta_1)\sin(\theta_2)}$ \\\\
$\displaystyle \fpd{u_2}{s_1} = \frac{\sin(2\chi - \theta_1)\sin(2\chi-\theta_2)-\sin(\theta_1)\sin(\theta_2)}{\ell_2\cos(\chi)}$ & \multirow{2}{*}{$\displaystyle \fpd{u_2}{u_1} = \frac{\sin(2\chi-\theta_2)-\kappa_2\ell_2\cos(\chi)}{\sin(\theta_1)}$} \\
$\displaystyle \;\;\;\;\;-\kappa_1\sin(2\chi-\theta_2) - \kappa_2\sin(2\chi-\theta_1) + \kappa_1\kappa_2\ell_2\cos(\chi)$ & \\ \\
\end{tabular}

%
\noindent Furthermore, $\det(DT_1) = 1$ and $\det(DT_2)=1$. 
\end{prop}

\begin{figure}[htb]\label{fixeds}
\includegraphics[width=0.7\textwidth]{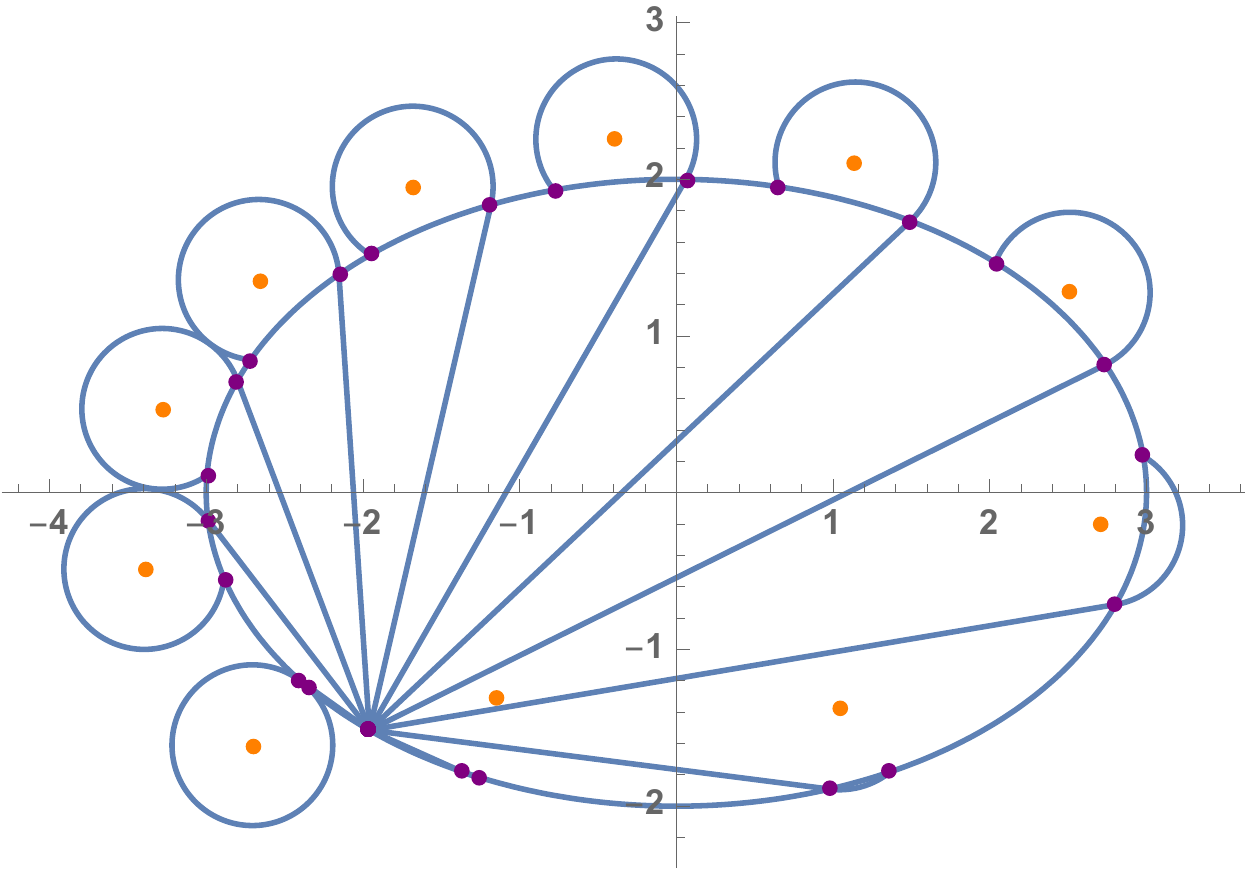}
\caption{The behavior of the return map for fixed $s_0$ and varying $u_0$ when $\mu < \rho_{min}.$ Also shown are the Larmor centers and points $P_1, P_2$ for each corresponding value of $u_0$.}
\end{figure}

 The details of this proof are given in Appendix \ref{jacpf}. The components of $DT_1$ are well-known while the components of $DT_2$ are analogous to those found in Proposition 1 of \cite{BK}.

\begin{cor}
Let $T = T_2 \circ T_1$. Then $DT = \twobytwo{\fpd{s_2}{s_0}}{\fpd{s_2}{u_0}}{\fpd{u_2}{s_0}}{\fpd{u_2}{u_0}}$ with 
\begin{align*}
\fpd{s_2}{s_0} &= \frac{\kappa_0\ell_1\sin(2\chi-\theta_1) - \sin(\theta_0)\sin(2\chi-\theta_1)-\kappa_0\ell_2\cos(\chi)\sin(\theta_1)}{\sin(\theta_1)\sin(\theta_2)} \\
\fpd{s_2}{u_0} &= \frac{\ell_1\sin(2\chi-\theta_1) - \ell_2\cos(\chi)\sin(\theta_1)}{\sin(\theta_0)\sin(\theta_1)\sin(\theta_2)} \\
\fpd{u_2}{s_0} &= \frac{\kappa_2\sin(\theta_0)\sin(2\chi-\theta_1)}{\sin(\theta_1)} + 
\frac{2\sin(\chi) \sin(2\chi-\theta_1-\theta_2) (\kappa_0\ell_1-\sin(\theta_0))}{\ell_2\sin(\theta_1)} \\
&\;\;\;\;\; -\kappa _0 \left(\sin(2\chi-\theta_2)+\frac{\kappa_2 \ell_1\sin(2\chi-\theta_1)}{\sin(\theta_1)}  - \kappa _2\ell_2\cos(\chi) \right)\\
\fpd{u_2}{u_0} &= \frac{ \kappa _2 \ell_2 \cos (\chi)-\sin(2\chi-\theta _2)}{\sin(\theta_0)}+ \frac{2\ell_1\sin(\chi )\sin(2\chi-\theta_1-\theta_2)-\kappa_2\ell_1\ell_2 \sin(2 \chi -\theta _1)}{\ell_2\sin(\theta_0)\sin(\theta_1)}.
\end{align*}
Furthermore, $\det(DT) =1$. 
\end{cor}

 From this we conclude that $T$ is an area- and orientation-preserving map of the annulus $\mathcal{P}$ and that the Birkhoff coordinates are conjugate. Just as with Birkhoff and magnetic billiards, the map $T$ preserves the symplectic area-form $ds\wedge du = \sin(\theta)ds \wedge d\theta$ on $\mathcal{P}$.

\section{Generating Functions and Twist Maps}
\label{genfun}


Twist maps have been studied extensively (\cite{Go}, \cite{Ma2}, \cite{Me}) in the context of dynamics and symplectic geometry. Let $f$ be a symplectic map from the annulus $\R{}/\Z \times \R{}$ to itself. To be a \emph{monotone twist map}, the lift of $f$ to its universal cover $\widehat{f}$ must satisfy the following properties, where $(x^\prime,y^\prime) = \widehat{f}(x,y)$:

\begin{enumerate}[i)]
\item $\widehat{f}(x+1,y) = \tilde{f}(x,y) + (1,0)$; 
\item $\fpd{x^\prime}{y} >0$ (twist condition);
\item $\widehat{f}$ admits a periodic exact symplectic map $G$ called a \emph{generating function}: $$y^\prime dx^\prime - y dx = dG(x,x^\prime).$$ Alternately we may say $y^\prime = \fpd{G}{x^\prime}$ and $y = -\fpd{G}{x}$. \\
\end{enumerate}


In Birkhoff billiards, the billiard map is always a monotone twist map whose generating function is the negative of the Euclidean (chord) distance between successive collisions with the boundary. In the magnetic billiard setting, the magnetic billiard map is not always twist, but when it is the generating function also depends upon the area associated with an arc of a given trajectory which appears as a flux term. It is not surprising that in this problem that has elements of both standard and magnetic billiards, that our generating function contains a combination of these elements.

To better understand when the return map $T$ is a twist map, we turn to the following theorem which we prove in Appendix \ref{twistcondgfpf}. 

\begin{theorem}\label{twistcondgf}
Let $\Gamma(s) = \partial \Omega$ be of class $C^k$, $k \geq 3$, and let $\rho_{min}$ be the minimum radius of curvature of the strictly convex boundary curve $\Gamma(s)$. Then if $\mu< \rho_{min}$ then $T$ is a twist map whose unique generating function (up to an additive constant) is given by $$G(s_0,s_2) = -\ell_1 - |\gamma| + \frac{1}{\mu} \mathcal{S}$$ where $\ell_1$ is the length of the line segment inside $\Omega$, $|\gamma|$ is the length of the circular arc $\gamma$ of Larmor radius $\mu$, and $\mathcal{S}$ is the area inside the circular arc $\gamma$ but outside $\Omega$. 
\end{theorem}

\begin{remark} This generating function need not be unique. But in general we can think of the generating function as the reduced action along a solution $\nu$ to the Euler Lagrange equations which connects $P_0$ to $P_2$. See \cite{B} and \cite{BK}. 
\end{remark}

%

 An interesting property of this generating function (and this problem in general) is as follows: In the high magnetic field limit (i.e. $\mu \to 0$), both $|\gamma| \to 0$ and $\frac{1}{\mu}\mathcal{S} \to 0$. This is because $|\gamma| = O(\mu)$ and $\mathcal{S} = O(\mu^2)$. So as $\mu \to 0$, our generating function approaches the standard billiard generating function, and our return map approaches the standard billiard map for billiards inside a convex set. \\

We can decompose $G$ into non-magnetic and magnetic parts, 
\begin{align*}
G(s_0,s_2) &= \left[-\ell_1 - \frac{1}{\mu}\mathcal{A}\right] + \left[-|\gamma| + \frac{1}{\mu} Area(\mathcal{A} \cup \mathcal{S})\right] \\
&= E(s_0,s_2) + F_\mu(\chi(s_0,s_2)).
\end{align*}
Here $Area(\mathcal{A} \cup \mathcal{S})$ is the area of $\mathcal{A} \cup \mathcal{S}$, $E(s_0,s_2)$ has quantities $\ell_1$ and $\mathcal{A}$ which are not directly dependent upon the magnetic field, 
and $F_\mu$ is dependent upon the magnetic field and can be written as 
$$F_\mu(\chi(s_0,s_2)) = -\mu( \chi + \sin(\chi)\cos(\chi)).$$  
We can also write $F_\mu$ as a function of $\ell_2$, though with caveats:  
$$F_\mu(\ell_2(s_0,s_2)) = -\mu\arccos\left(\pm \sqrt{1-\frac{\ell_2^2}{4\mu^2}}\right) - \pm \frac{\ell_2}{2}\sqrt{1-\frac{\ell_2^2}{4\mu^2}},$$ where $(+)$ is used if $0 < \chi \leq \frac{\pi}{2}$ and $(-)$ is used if $\frac{\pi}{2}< \chi < \pi$. 

\section{Periodic Orbits}\label{PerOrbs}

The study of periodic orbits and their properties is a fundamental part of any dynamical system. In billiards, Birkhoff used Poincar{\'e}'s last geometric theorem to show the existence of infinitely many distinct orbits (\cite{Bir}). One way to distinguish distinct periodic orbits from one another is by the \emph{rotation number}. The rotation number of a periodic orbit is the rational number $$\frac{m}{n} = \frac{\text{winding number}}{\text{minimal period}} \in [0,1]$$ where the winding number $m>1$ is computed with respect to the orientation of $\partial \Omega$ induced by the parametrization $\Gamma(s)$. A periodic orbit with rotation number $\frac{m}{n}$ is sometimes referred to as having \emph{frequency} $(m,n)$.


%

A continuous orientation-preserving homeomorphism of the circle $S^1$ to itself has a well-defined rotation number, defined modulo 1, when the circle is normalized to have perimeter 1. When a lift to $\R{}$ of this homeomorphism is chosen, this rotation number is now a real number. By the definition of a twist map, $T$ sends boundary circles to boundary circles, so the lifted homeomorphism has a bottom and top rotation number, $\omega_-$ and $\omega_+$. Then the rotation numbers belong to an interval $\mathcal{I}(\widehat{T}) = [\omega_-,\omega_+]$ provided $\omega_- < \omega_+$. In particular, if the map is the identity on the boundary circles then necessarily $\omega_-,\omega_+ \in \Z$. 

With this idea in mind, we can extend our definition of rotation number of the orbit $\{(s_{2k},u_{2k})\}_{k \in \Z}$ to include irrational numbers by writing $$\omega = \frac{1}{L}\limit{k}{\infty}{\frac{s_k}{k}},$$ provided this limit exists. We note that in the context of Birkhoff billiards and inverse magnetic billiards, this definition agrees with the geometric definition in terms of winding number given above. \\

\begin{figure}[t]
\begin{tabular}{ l l } 
(a) \includegraphics[width=0.40\textwidth]{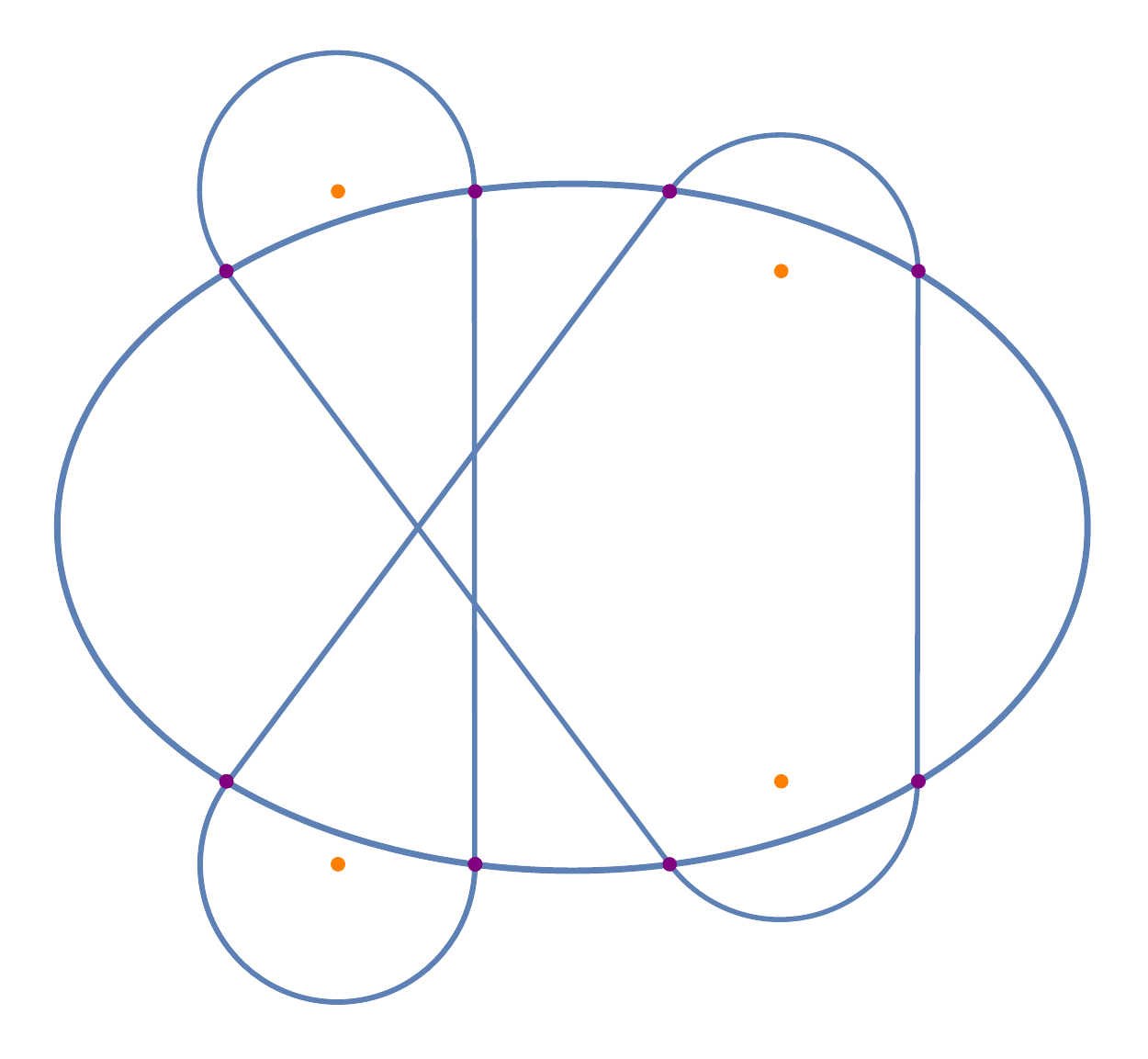} & (b) \includegraphics[width =0.45\textwidth]{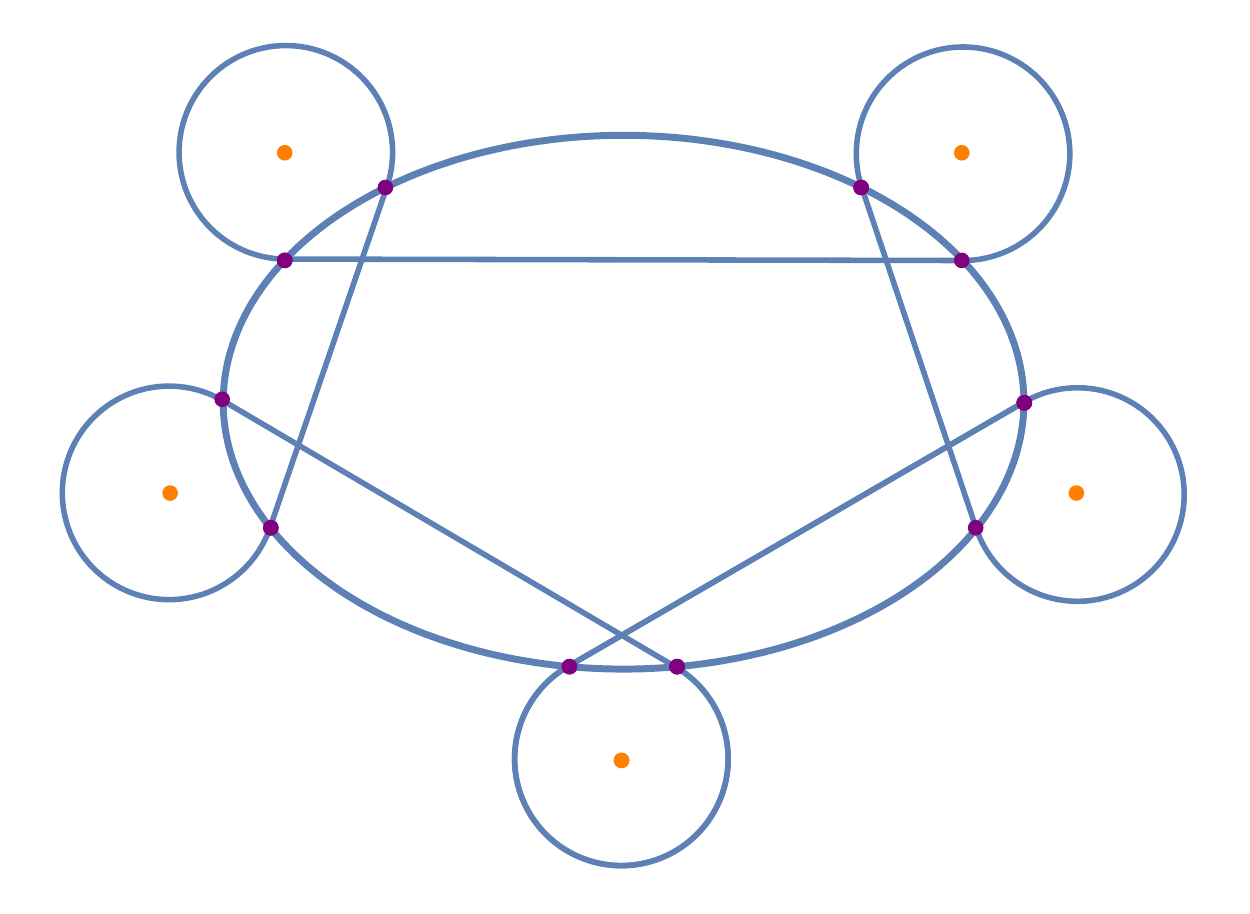} \\
\end{tabular}
\caption{A $(2,4)$ (a) and $(4,5)$ (b) periodic orbit in an ellipse with $\mu < \rho_{min}.$ The centers of the Larmor circles are marked in orange and the points $P_i$ are in dark purple.}
\label{twoellipseexamples}
\end{figure}

One particularly useful application of a generating function is in the search of periodic orbits. 
We will denote the lift of the return map $T$ as $\widehat{T}$. Theorems about the existence of periodic orbits for continuous area-preserving twist maps can be attributed to Poincar{\'e} and Birkhoff, Aubry, Mather, and Meiss and MacKay. These theorems describe the existence of $(m,n)$ periodic orbits that are ``maximizing" and ``maximin" when considering the quantity $\sum_{j=k}^{l-1} G(s_{2j}, s_{2j+2})$ along with characterizations of irrational $\omega \in \mathcal{I}(\widehat{T})$ and their relationship to quasiperiodic orbits.   See, for example, \cite{Me} or the summary from section 4 of \cite{BK} for more details. 

We take a similar approach below, and can apply the theorems Poincar{\'e}, Birkhoff, Aubry, and Mather to the map $T$ while making qualitative comments about the behavior of $T$. 

\begin{prop}
Consider the three curvature regimes:
\begin{enumerate}
\item If $\mu < \rho_{min}$, the function $s_2(s_0, u_0)$ is strictly monotonic in $u_0$ and $T$ is a twist map. For fixed $s_0$ the curve $\{T(s_0,u_0): -1 <u_0 <1 \}$ rotates once around phase space (see Figure \ref{phasespaces}a) with $\limit{u_0}{\pm 1}{T(s_0,u_0)} = (s_0,u_0)$. Therefore $\mathcal{I}(\widehat{T}) = [0,1]$. 
\item If $\rho_{min} < \mu < \rho_{max}$, then the map may be discontinuous due to the Larmor circle becoming tangent to the boundary. The function $s_2(s_0,u_0)$ is not necessarily monotonic in $u_0$ and is not a twist map (see Figure \ref{phasespaces}c). It is still true that  $\limit{u_0}{1}{T(s_0,u_0)} = (s_0,u_0)$, but not necessarily when $u_0 \to -1$. 
\item If $\rho_{max} < \mu$, then $s_2(s_0,u_0)$ is initially decreasing in $u_0$ and then begins to increase again (see Figure \ref{phasespaces}e). We still have $\limit{u_0}{\pm 1}{T(s_0,u_0)} = (s_0,u_0)$, which implies that there are exactly two distinct trajectories with equal $\chi$ for a given $s_0$, $s_2$. 
\end{enumerate}
\label{rotnum}
\end{prop}


\begin{figure}[hp]
\begin{tabular}{ l l }
(a) \includegraphics[width=0.4\textwidth]{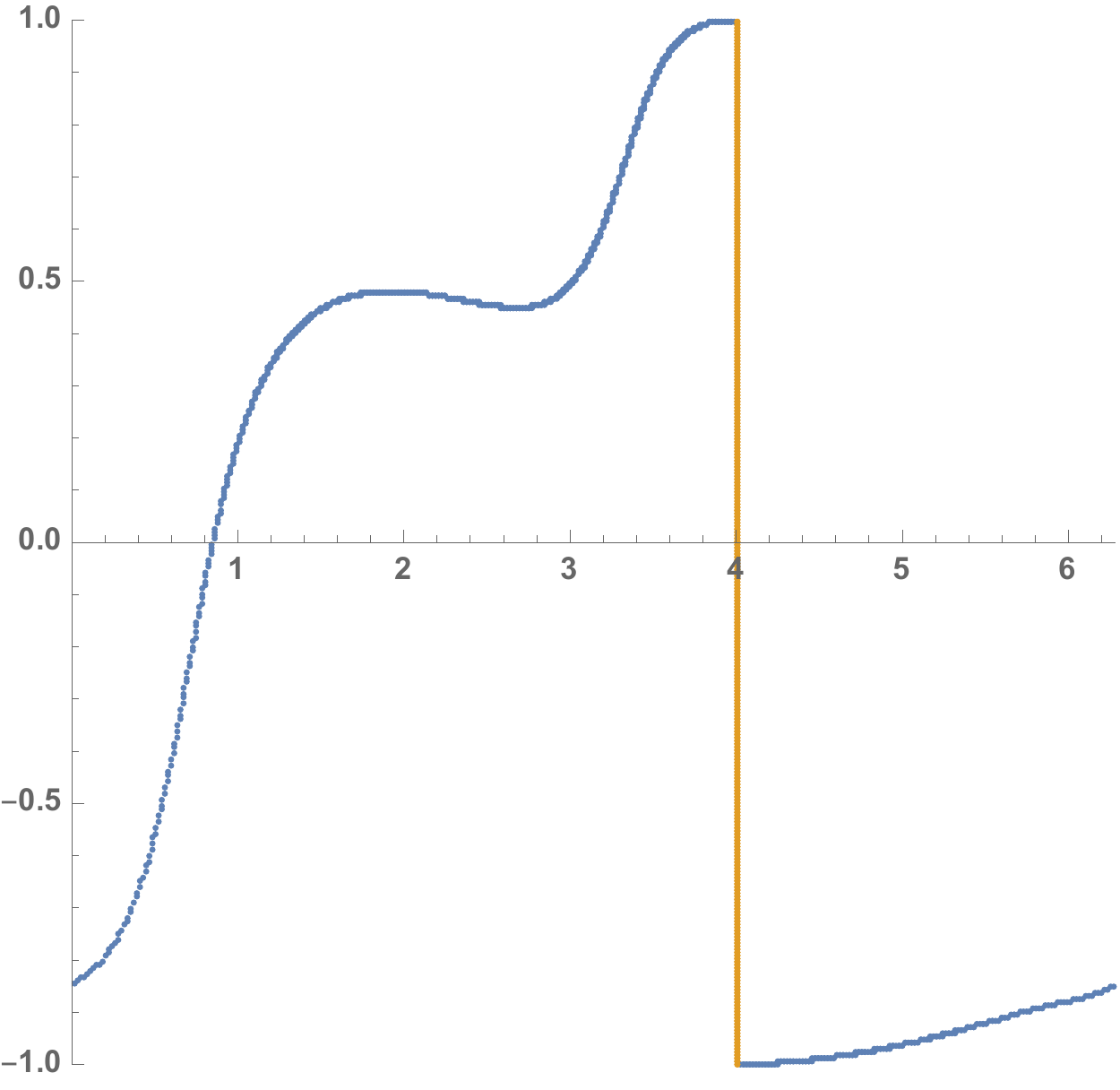} & (b) \includegraphics[width=0.42\textwidth]{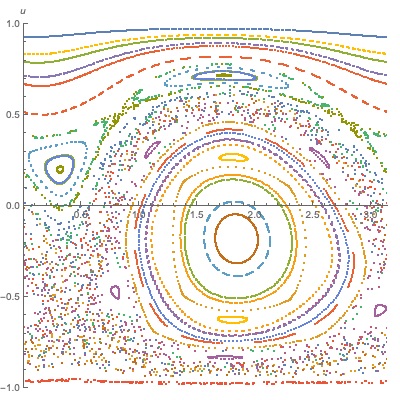} \\
(c) \includegraphics[width=0.4\textwidth]{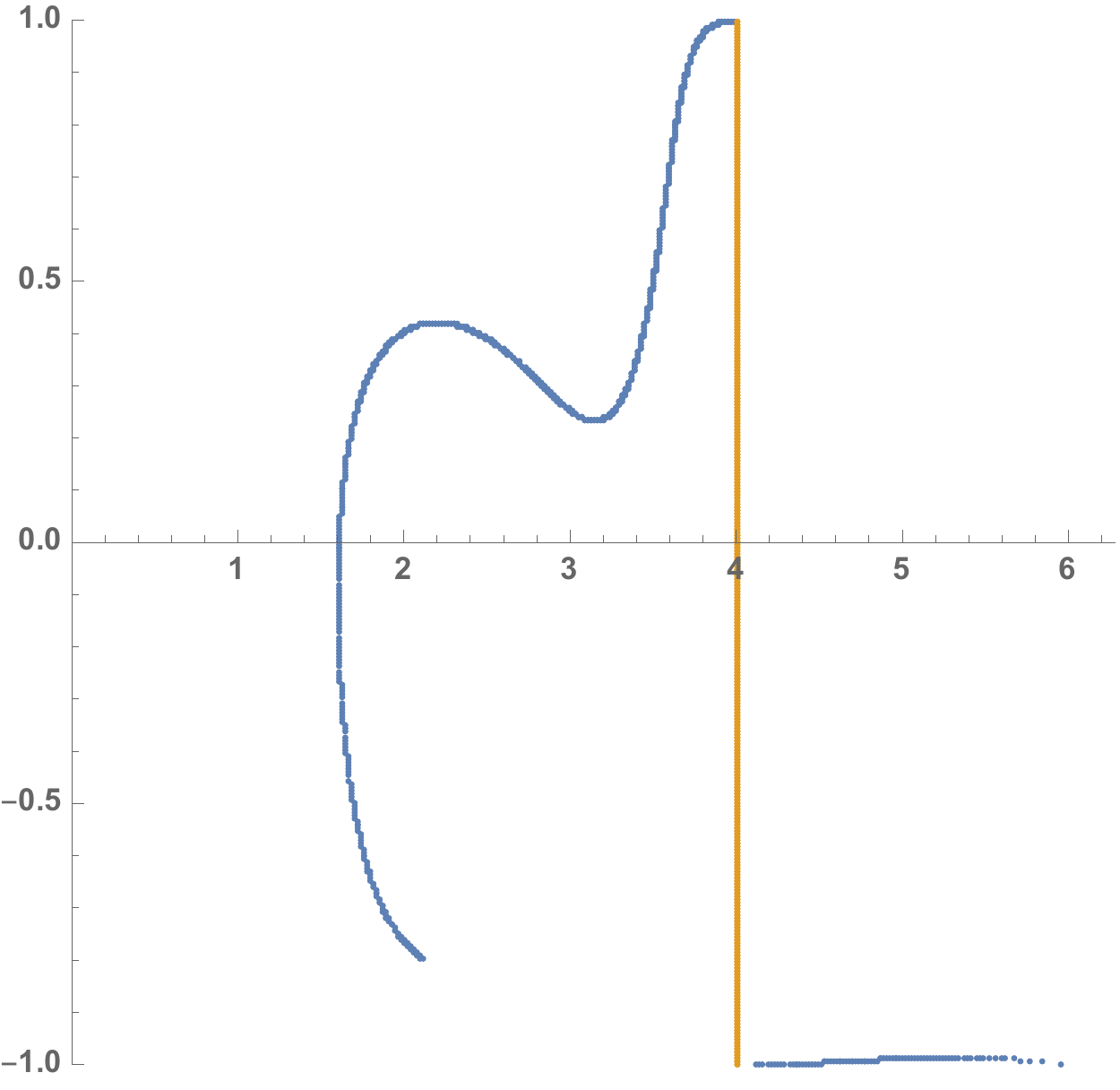} & (d) \includegraphics[width=0.42\textwidth]{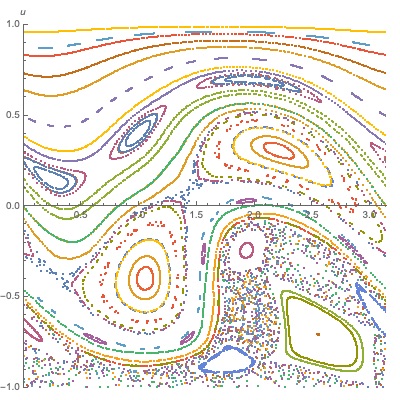} \\
(e) \includegraphics[width=0.4\textwidth]{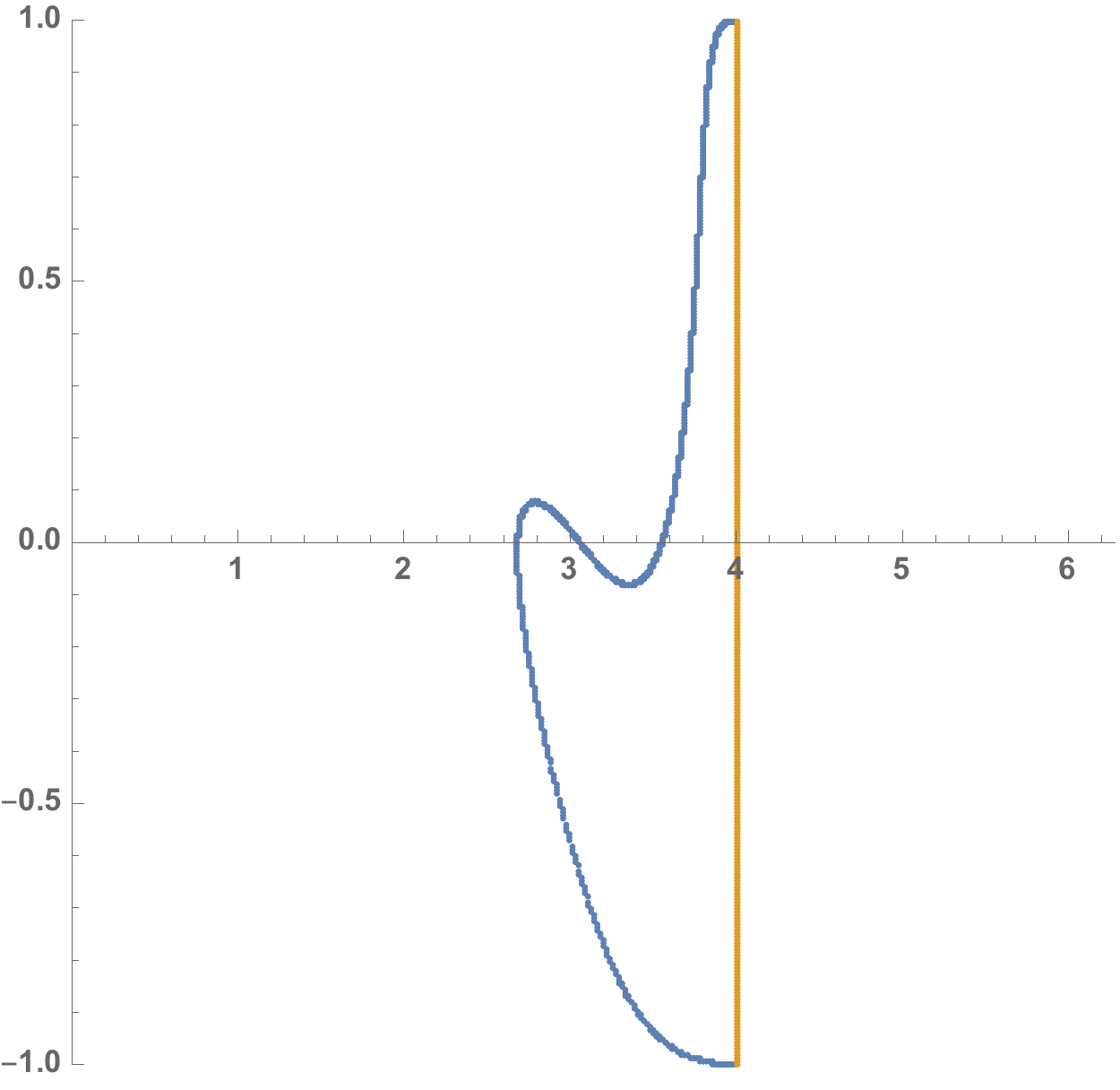} & (f) \includegraphics[width=0.42\textwidth]{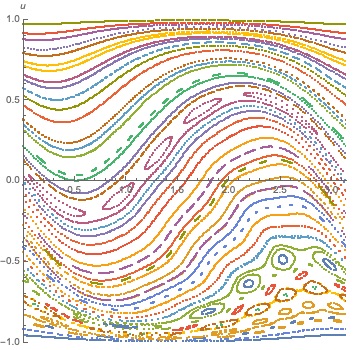}
\end{tabular}
\caption{Structure of half the phase space in the $(\phi,u)$-plane for an ellipse where horizontal axis is $\phi$,  the polar angular parameter used in place of arc length, $s$. The vertical line $\phi = \phi^*$ is shown with its image under $T$ in the left column, $-1 \leq u \leq 1$ for: (a) $\mu < \rho_{min}$; (c) $\rho_{min}< \mu < \rho_{max}$; (e) $\rho_{max}< \mu$. The right column is half of a typical phase portrait, $0 \leq \phi \leq \pi$, of an ellipse for: (b) $\mu < \rho_{min}$; (d) $\rho_{min}< \mu < \rho_{max}$; (f) $\rho_{max}< \mu$. }
\label{phasespaces}
\end{figure}

 This proposition is very similar to the qualitative behavior of magnetic billiards (c.f. section 5.1 of \cite{BK}). In particular, part (1) tells us that for any convex set with smooth boundary and nonvanishing curvature, periodic orbits of every rational frequency $\omega = \frac{m}{n}$ exist, and the earlier summary gives us information about rational and irrational orbits in $\mathcal{I}(\widehat{T})$.

\section{When $\partial\Omega$ is an Ellipse}\label{OmegaIsEllipse}

 We take a quick detour and consider the case when $\partial\Omega$ is an ellipse. Consider the parametrization of $\partial\Omega$ as $$ \bs{x}(\phi) = (\lambda \cos(\phi), \sin(\phi)), \;\;\;\;\; \frac{ds}{d\phi} = C(\phi) = \sqrt{\lambda^2\sin^2(\phi) + \cos^2(\phi)}.$$ Without loss of generality we may assume that the parameter $\lambda \geq 1$. In such a case, $\rho_{min} =\lambda^{-1}$. Consider the points $P_i = \bs{x}(\phi_i), i=0,1,2$. 

 Assuming then that $\mu <\rho_{min}$, an important geometric consequence is that $\phi_2-\phi_1 < \pi$, which simplifies the calculation below. Then $T$ is a twist map and $$G = -2\sin(\phi_{10}^-)C(\phi_{10}^+)- \frac{1}{\mu}\lambda \left(\phi_{21}^- -\frac{1}{2}\sin(2 \phi_{21}^-)\right) + F_{\mu}(2\sin(\phi_{21}^-)C(\phi_{21}^+))$$ where $\phi_{ab}^\pm = \dfrac{\phi_a \pm \phi_b}{2}$. 


In the case that $\Omega$ is the unit disk ($\lambda =1$), we see that $C=1$ and hence $\phi_2 - \phi_0 = 2\chi$. Another geometric observation is that $\theta_i=\theta$ and $\chi$ are both constant, and hence $u_i = u$ is constant. This is because the the diagram in Figure \ref{stdpic1} is symmetric about the line connecting the center of $\Omega$ and the center of the circular arc. This in turn implies that all of our geometric quantities, $\ell_1$, $\ell_2$, $|\gamma|$, and $\mathcal{S}$ are constant as they only depend upon $\theta$ and $\chi$. Using the elementary geometry of a circle-circle intersection,  
$$ \chi = \theta + \arcsin\left( \frac{\mu \sin(\theta)}{\sqrt{1+\mu^2 - 2\mu \cos(\theta)}} \right) $$ 
and the return map is explicitly $$T(s,u) = (s+2\chi,u).$$ It is clear that since $\theta$ is constant, $u$ is a constant of motion and the system is integrable (in the sense of Liouville). 
Further, the simplicity of the return map in the circular case allows us to find periodic orbits directly. Since $$T(s_0,u_0) = (s_0+2\chi, u_0),$$ we see that $$\widehat{T}^{n}(s_0,u_0) = (s_0 + 2n\chi,u_0).$$ Therefore a periodic orbit will have rotation number $m/n$ if and only if  
$\chi = m\pi/n.$ 


\begin{figure}[p]
\begin{tabular}{ c c }  
(a) \includegraphics[width =0.45\textwidth]{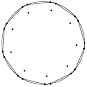} & (b) \includegraphics[width =0.45 \textwidth]{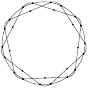} \\
(c) \includegraphics[width =0.45\textwidth]{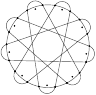} & (d) \includegraphics[width =0.45\textwidth]{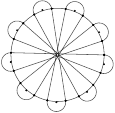} \\  
(e) \includegraphics[width =0.45\textwidth]{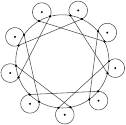} & (f) \includegraphics[width =0.45\textwidth]{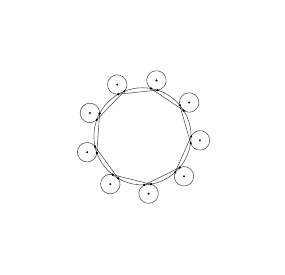} \\
\end{tabular}
\caption{Periodic orbits of period 9 in the unit circle with $\mu<\rho_{min}$: (a)  $(1,9)$ orbit; (b) $(2,9)$ orbit; (c) $(4,9)$ orbit; (d) $(5,9)$ orbit; (e) $(7,9)$ orbit; (f) $(8,9)$ orbit. The dots along the circle are the points $P_i$ while the other dots are the centers of the Larmor arcs.}
\label{period 9 orbits}
\end{figure}

\section{Existence and Nonexistence of Caustics}\label{CausticsChapter}
\subsection{Preliminary Results}

In both standard and magnetic billiards, the question of the existence of caustics has been addressed by Lazutkin \cite{L}, Berglund and Kunz \cite{BK}, Moser \cite{Mo2}, \cite{Mo1}, and more in several variants of the standard billiard problem. Trajectories with caustics (the ``whispering gallery modes'') correspond to invariant curves (a homotopically nontrivial curve) in phase space. Lazutkin had to assume a high degree of differentiability of the boundary in order to guarantee the existence of caustics, though this was later reduced to degree 6 by Douady \cite{D}. 

Due to the nature of inverse magnetic billiards problem, we call a smooth closed convex curve in $\Omega$ with the property that each trajectory that is tangent to it stays tangent to it after each successive reentry an \emph{inner convex caustic}. An analogous definition holds for \emph{outer} caustics which contain $\Omega$. 

For example, in a circle of radius $R$, elementary geometry shows that all trajectories of the inverse magnetic billiard have both inner and outer caustics that are circles of radii $r_{inner} = R|\cos(\theta_0)| = R|u_0|$ and $r_{outer} = \mu+\sqrt{R^2 + \mu^2 - 2R \mu \cos(\theta_0)}$, respectively. All of the trajectories in Figure \ref{period 9 orbits} have both inner and outer circular caustics.  \\

Our first result in this regard is an inverse magnetic version of Mather's theorem (\cite{Ma1}, \cite{Ma2}): If a billiard table with a smooth convex boundary curve has a point of vanishing curvature, then the billiard inside the curve has no caustics. 

\begin{theorem}
If the boundary of the billiard table $\partial\Omega$ has a point of vanishing curvature and $\mu <\rho_{min}$, the inverse magnetic billiard has no interior caustics. 
\end{theorem}

\begin{proof}
By Birkhoff's Theorem (\cite{Bir}), an invariant curve of an area-preserving twist map is a graph of some function. If our billiard has a caustic, then we have a one-parameter family of chords $P_0P_1$ to $\Gamma$ corresponding to points on the invariant curve. The graph property of Birkhoff's Theorem implies that if $P_0^*P_1^*$ is a nearby chord such that $P_0^*$ has moved along $\Gamma$ in the positive direction from $P_0$ then $P_1^*$ has moved in the positive direction from $P_1$ on $\Gamma$. These chords must intersect in the interior of $\Gamma$, and by the existence of the caustic, must be tangent to the caustic. 

\begin{figure}[htb]
\includegraphics[width=0.7\textwidth]{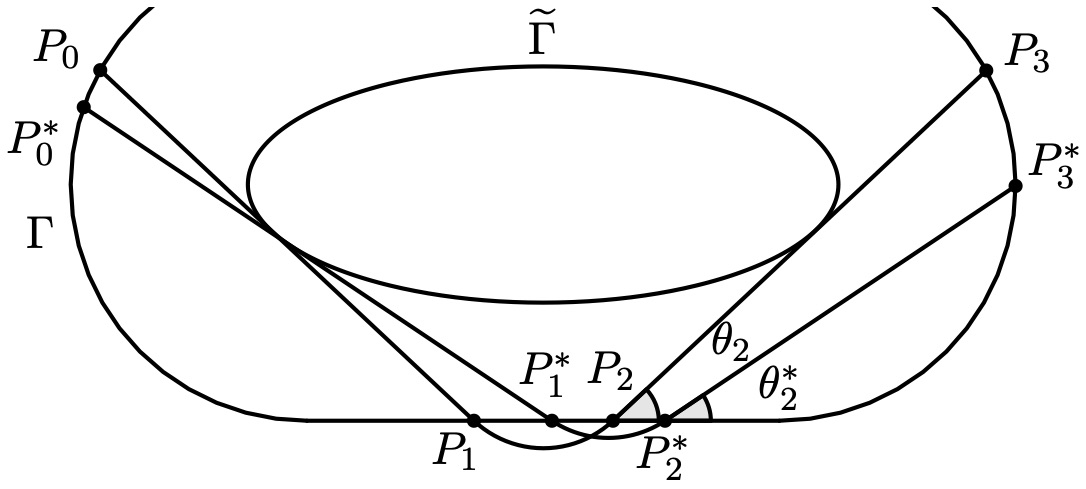}
\caption{Picture of the proof of the nonexistence of caustics if the boundary has a point of vanishing curvature.}
\label{vancurv}
\end{figure}

Assume an interior caustic $\widetilde{\Gamma}$ exists. The billiard portion of a trajectory forms a chord $P_0P_1$ tangent to the caustic, moves along its magnetic arc, and reenters $\Omega$ to form the next chord $P_2,P_3$, also tangent to the caustic. Suppose the curvature $\Gamma$ as $P_2$ vanishes. Consider an infinitesimally close chord $P_0^*P_1^*$, tangent to the same caustic, as described earlier, along with its next chord $P_2^*P_3^*$. Since the curvature at $P_2$ vanishes, the tangent line at $P_2^*$ is, in the linear approximation, the same as the one at $P_2$. Let $\theta_2$ and $\theta_2^*$ be the angle between the linear approximation and the chords $P_2P_3$ and $P_2^*P_3^*$, respectively.

There are three geometrically distinct cases. If $\chi< \frac{\pi}{2}$, then $\theta_2 > \theta_2^*$, and so the chords $P_2P_3$ and $P_2^*P_3^*$ will not intersect in the interior of $\Gamma$, a contradiction. See Figure \ref{vancurv}. Similarly, if $\chi > \frac{\pi}{2}$, $\theta_2 > \theta_2^*$. And if $\chi = \frac{\pi}{2}$, then the chords $P_2P_3$ and $P_2^*P_3^*$ are parallel and will not intersect. 
\end{proof}

To better understand the nature of caustics in this inverse magnetic billiard setting, we seek to understand the maps $T_1$ and $T_2$ near the boundary, as they show qualitatively different behavior. We also make the adjustment to the maps $T$, $T_1$, and $T_2$ so they are defined on the annulus $\R{}/L\Z\times [0,\pi]$ so the second variable is $\theta_i$ instead of $u_i$. \\

Lazutkin produced a well-known calculation of the Taylor expansion of the billiard map $T_1$ up to fourth order in $\theta$ (\cite{L}) and Berglund and Kunz calculate the Taylor expansion of the inner magnetic billiard map $T_2^*$ up to first order in $\theta$ (\cite{BK}). While Lazutkin proved the existence of a positive measure set of caustics sufficiently close to the boundary, Berglund and Kunz show the existence of caustics in inner magnetic billiards for three special cases by citing a version of the KAM theorems (\cite{Mo1}, \cite{D}). 

\subsection{Mimicking the Approach of Berglund and Kunz}\label{BKStyleCaustics}

We can investigate the behavior of the outer magnetic billiard map $T_2$ using the same techniques in \cite{BK}, and ultimately learn about $T$. For a nonzero magnetic field near the boundary, we will be able to apply KAM theorems to show the existence of invariant curves.

By adapting the proof of the Taylor expansion of the inner magnetic billiard map $T_2^*$ in section 5.2 of \cite{BK} to the outer magnetic billiard map $T_2$, we arrive at a similar expression. Coefficients are also calculated in appendix \ref{taylorexppf} directly. We state this result in the following proposition.

\begin{prop}\label{BKMagMap}
If the boundary $\partial \Omega$ is $C^k$, the outer magnetic billiard map $T_2$ is $C^{k-1}$ for small $\sin(\theta_1)$ and has the form 
\begin{align*}
s_2 &= s_1 + \frac{2\mu \sin(\theta_1)}{1-\mu \kappa_1 \cos(\theta_1)} + o(\sin(\theta_1)) \mod L \\
\theta_2 &= \theta_1 + o(\sin(\theta_1)). 
\end{align*}
Therefore, near $u = -1$, the map is of the form
\begin{align*}
s_{2} &= s_1 + \frac{2\mu}{1-\mu \kappa_1 }\theta_1 + o(\theta_1) \mod L\\
\theta_{2} &= \theta_1 + o(\theta_1) \\
\end{align*}
and near $u=1$, writing $\theta_i = \pi-\eta_i$ the map is of the form
\begin{align*}
s_{2} &= s_1 + \frac{2\mu}{1+\mu \kappa_1 }\eta_1 + o(\eta_1) \mod L \\
\eta_{2} &= \eta_1 + o(\eta_1). 
\end{align*}
\end{prop}


 We must be cautious as there are two properties we must check with regards to the map above. First, the map must be well-defined (i.e. the denominators may not vanish). This is only an issue when $\theta \ll 1$. The second is that the outer magnetic billiard map must denote the \textit{correct} intersection of the magnetic arc with the boundary of our billiard table. This is only an issue if a magnetic arc intersects the boundary in more than two places.
 
\begin{defn}
A closed $C^k$, $k\geq 2$, planar curve $\Gamma$ is said to have the \textbf{$\bs{\mu}$-intersection property} for some $\mu > 0$ if any circle of radius $\mu$ intersects $\Gamma$ at most twice.
\end{defn}

 However, a sufficient condition for the $\mu$-intersection property to be satisfied is for either $\mu < \rho_{min}$ or $\rho_{max} < \mu$ (Corollary to Lemma 3 in Appendix D of \cite{BK}). When satisfied, there is a one-to-one correspondence between inner magnetic billiard trajectories and outer magnetic billiard trajectories: For every outer magnetic arc there is a ``dual trajectory'' that is the complementary arc which completes the Larmor circle. This complementary arc can be interpreted as an inner magnetic billiard map with no change to our magnetic field convention. See Figure \ref{duality}. 

\begin{figure}[h]
\includegraphics[width=0.5\textwidth]{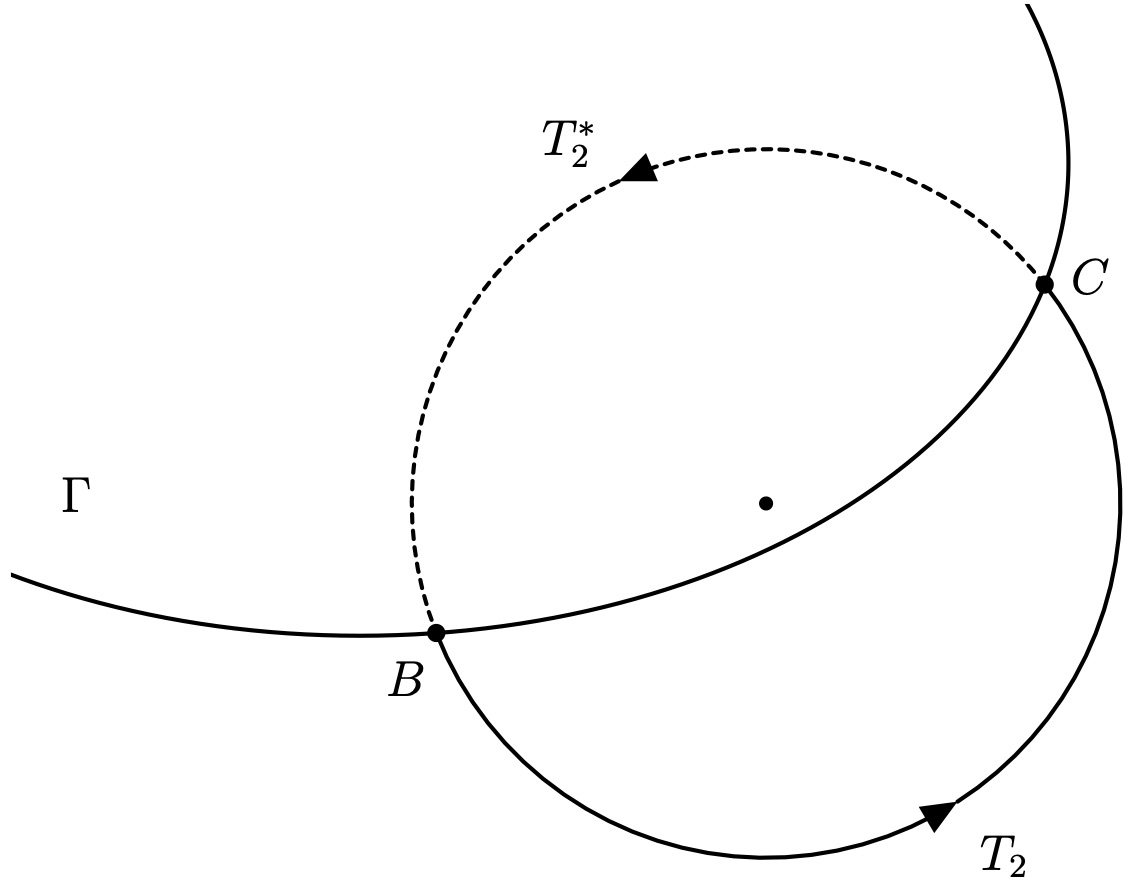} \\
\caption{The duality of the magnetic billiard map: every inner magnetic billiard trajectory has a corresponding outer magnetic billiard trajectory, provided the $\mu$-intersection property is satisfied. }
\label{duality}
\end{figure}

Therefore determining the correct intersection point from our map is only an issue when $\rho_{min} < \mu < \rho_{max}$, as a Larmor circle in this case may intersect $\partial\Omega$ in more than two places. 


If we consider the three curvature regimes, we notice the following:
\begin{enumerate}
\item If $\mu < \rho_{min}$, then $\mu \kappa(s) \leq \mu \kappa_{max} <1$, so $0 < 1-\mu \kappa(s)$ for all $s$;
\item If $\rho_{min} < \mu < \rho_{max}$, then $\frac{\kappa_{min}}{\kappa_{max}} < \mu\kappa_{min} < 1 < \mu \kappa_{max} < \frac{\kappa_{max}}{\kappa_{min}}$;
\item If $\rho_{max}< \mu$, then $1 < \kappa_{min}\mu \leq \kappa(s)\mu$, so $1-\mu \kappa(s) <0$ for all $s$.
\end{enumerate}
The denominators of the coefficients in the theorem above are well-defined in cases (1) and (3), but potentially not defined in case (2). 

\begin{prop}\label{taylorexp}
The inverse magnetic billiard map $T$ admits the following Taylor expansion for $\theta_i$ near 0:
\begin{align*}
s_{i+2}&= s_i + \frac{2}{\kappa_i(1-\mu \kappa_i)}\theta_i +O (\theta_i^2) \mod L\\
\theta_{i+2} &= \theta_i + O(\theta_i^2) \\
\end{align*}
 where we have omitted the dependence upon $s_i$ and that $\kappa_i := \kappa(s_i) \approx \kappa(s_{i+1})$.

\noindent For $\theta_i = \pi-\eta_i$ with $\eta_i$ near 0, the map $T$ admits the Taylor expansion
\begin{align*}
s_{i+2} &= s_i - \frac{2}{\kappa_i(1+\mu \kappa_i)}\eta_i + O(\eta_i^2) \mod L \\
\eta_{i+2} &= \eta_i + O(\eta_i^2). 
\end{align*}
 \end{prop}
 
An outline of the proof is given in appendix \ref{taylorexppf}. First we observe that both versions of this map are not well-defined if the curvature is allowed to vanish, which is consistent with our version of Mather's result. Further, consider the two limiting cases of $\mu$: if $\mu \to \infty$, the map $T$ limits to $s_2  = s_0 + O(\theta_0^2)$, and $\theta_2 = \theta_0 + O(\theta_0^2)$ which the identity map to first order; And if $\mu \to 0^+$, the map $T$ limits to $s_2 = s_0 + \frac{2}{\kappa_0}\theta_0 + O(\theta_0^2)$, which is the standard billiard map to first order. This is consistent with the geometric observations via the generating function in Section \ref{genfun}. \\

We may now make comments about the maps above in the style of \cite{BK}. 

\begin{enumerate}
\item Near $u = -1$, the map $T$ has the form
\begin{align*}
s_{i+2}&= s_i + \frac{2}{\kappa_i(1-\mu \kappa_i)}\theta_i +O (\theta_i^2) \mod L\\
\theta_{i+2} &= \theta_i + O(\theta_i^2). 
\end{align*}
We have already observed that the denominator will not vanish in two cases:
\begin{itemize}
\item If $\mu < \rho_{min}$, then we make the change of variables $\varphi_i = s_i - \mu \tau_i$ and $r_i =2\rho_i \theta_i$ to make the map of the form
\begin{align*}
\varphi_2 &= \varphi_0 + r_0 + O(r_0^2) \mod L -2\pi\mu \\
r_2 &= r_0 + O(r_0^2). 
\end{align*}
This corresponds to the correct intersection of the magnetic arc with the boundary, as this trajectory corresponds to a small billiard chord forward plus a small skip forward along the boundary from the outside. 
\item If $\rho_{max}< \mu$, then we make the change of variables $\varphi_i = \mu \tau_i - s_i$ and $r_i = 2\rho_i\theta_i$ to make the map of the form 
\begin{align*}
\varphi_2 &= \varphi_0 - r_0 + O(r_0^2) \mod 2\pi\mu - L \\
r_2 &= r_0 + O(r_0^2). 
\end{align*}
Again, this is the correct intersection with the boundary, because a large magnetic arc will reenter $\Omega$ ``behind'' its exit point. 
\end{itemize}
\item Near $u = 1$, the map $T$ has the form
\begin{align*}
s_{i+2} &= s_i - \frac{2}{\kappa_i(1+\mu \kappa_i)}\eta_i + O(\eta_i^2) \mod L \\
\eta_{i+2} &= \eta_i + O(\eta_i^2) 
\end{align*}
where we have written $\theta_i = \pi - \eta_i$. Observe that the denominator can never vanish, so this approximation is valid for all three curvature regimes. Moreover, this map can be understood as a short interior billiard chord backwards followed by most of a circular magnetic arc forwards, ultimately resulting in traveling a small distance backwards. The change of variables $\varphi_i = s_i + \mu \tau_i$ and $r_i = 2\rho_i \eta_i$ turns the map into 
\begin{align*}
\varphi_2 &= \varphi_0 - r_0 + O(r_0^2) \mod L + 2\pi\mu \\
r_2 &= r_0 + O(r_0^2). 
\end{align*}
\end{enumerate}

\begin{figure}[p]
\begin{tabular}{ l l }  
(a) \includegraphics[width =0.50\textwidth]{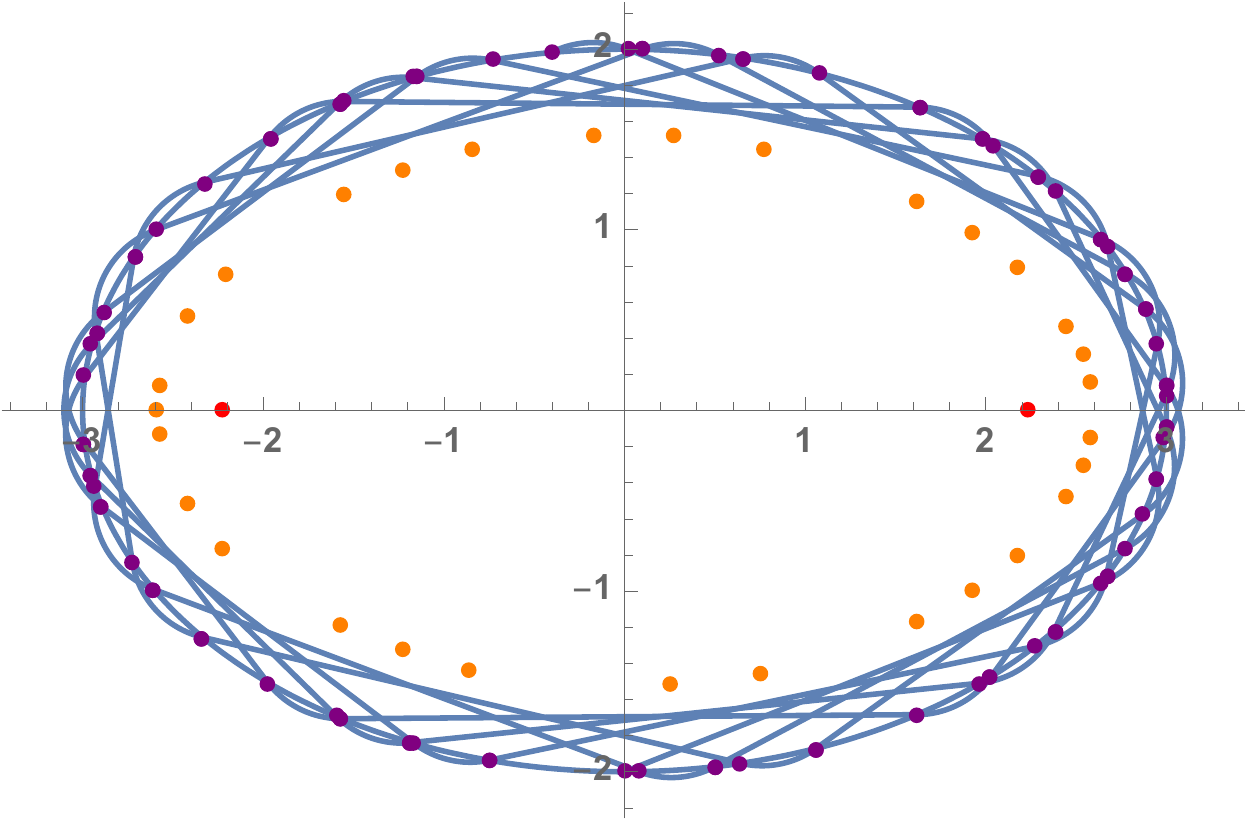} & (b) \includegraphics[width =0.45\textwidth]{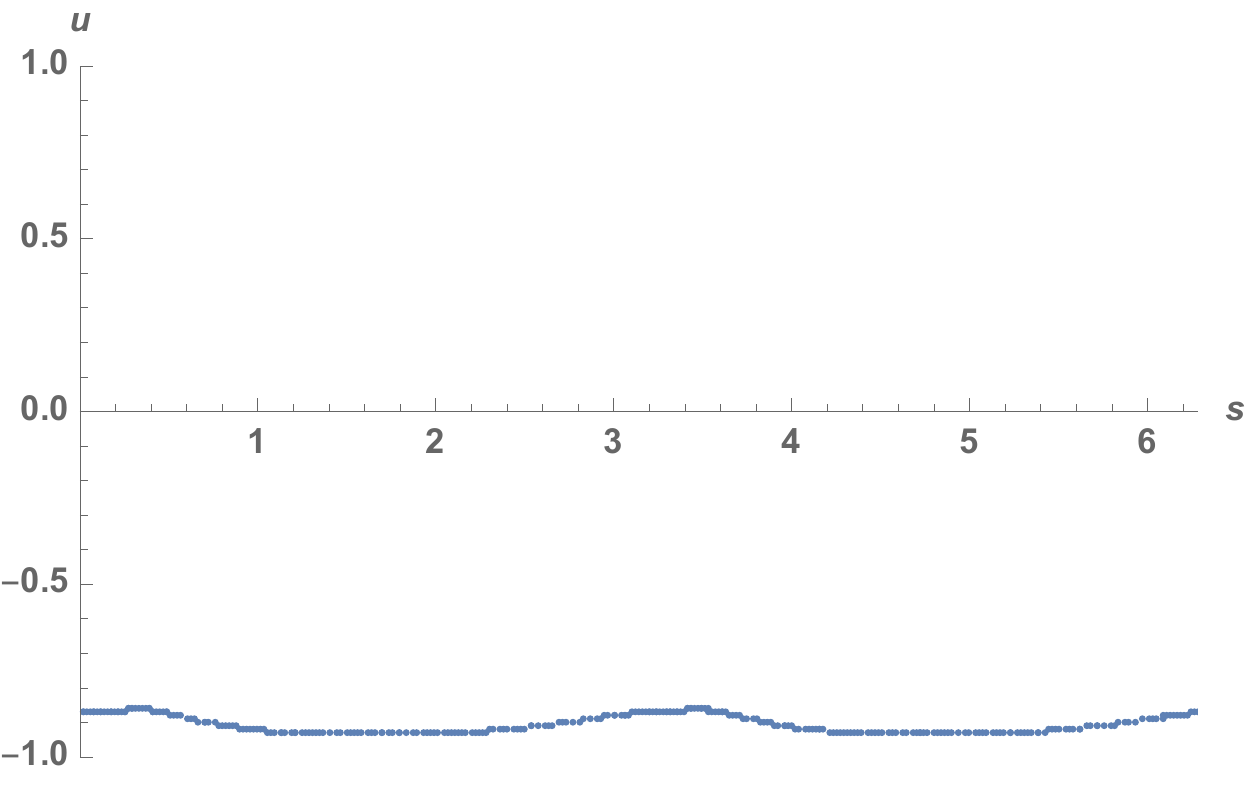} \\
(c) \includegraphics[width =0.50\textwidth]{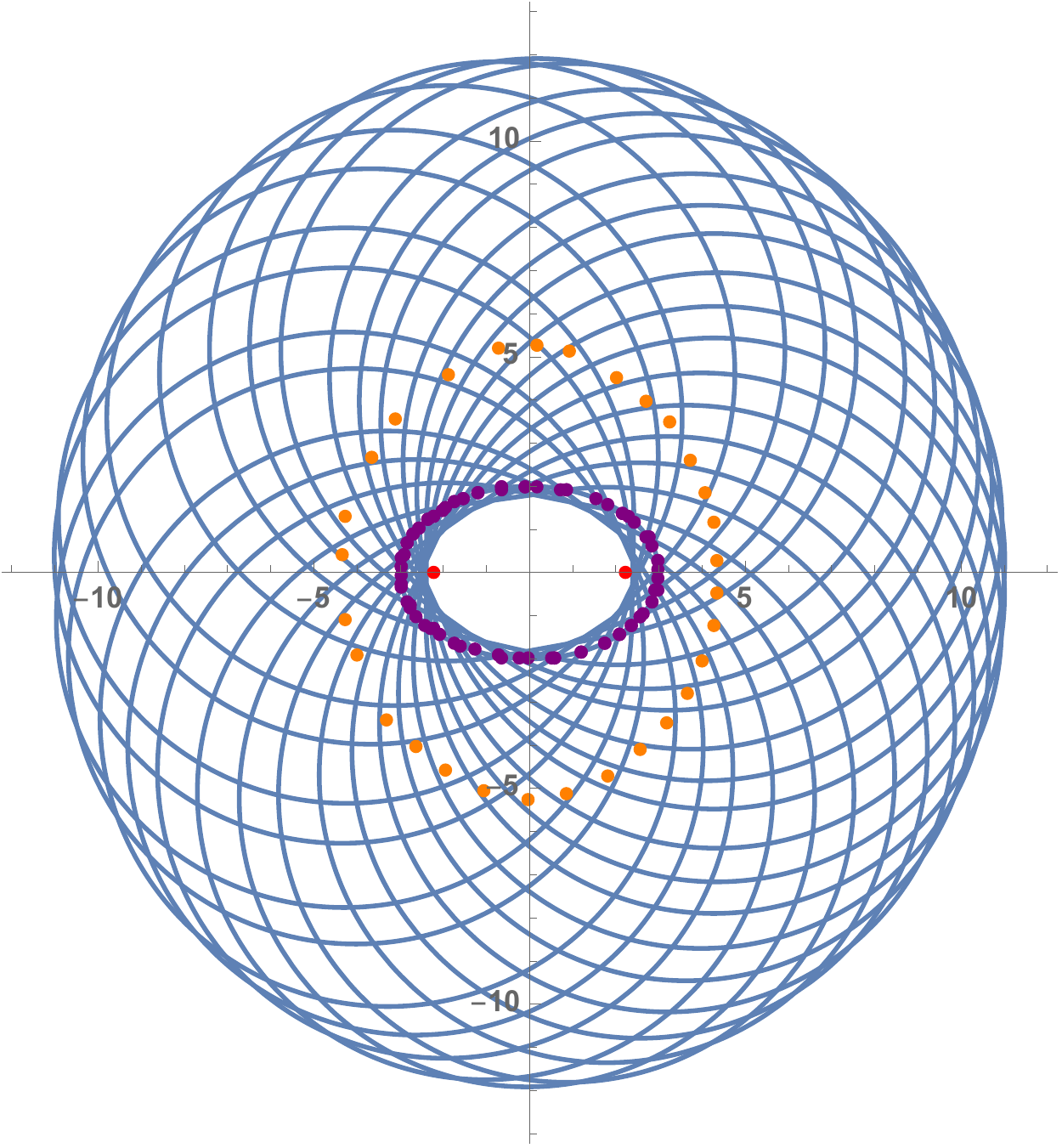} & (d) \includegraphics[width =0.45\textwidth]{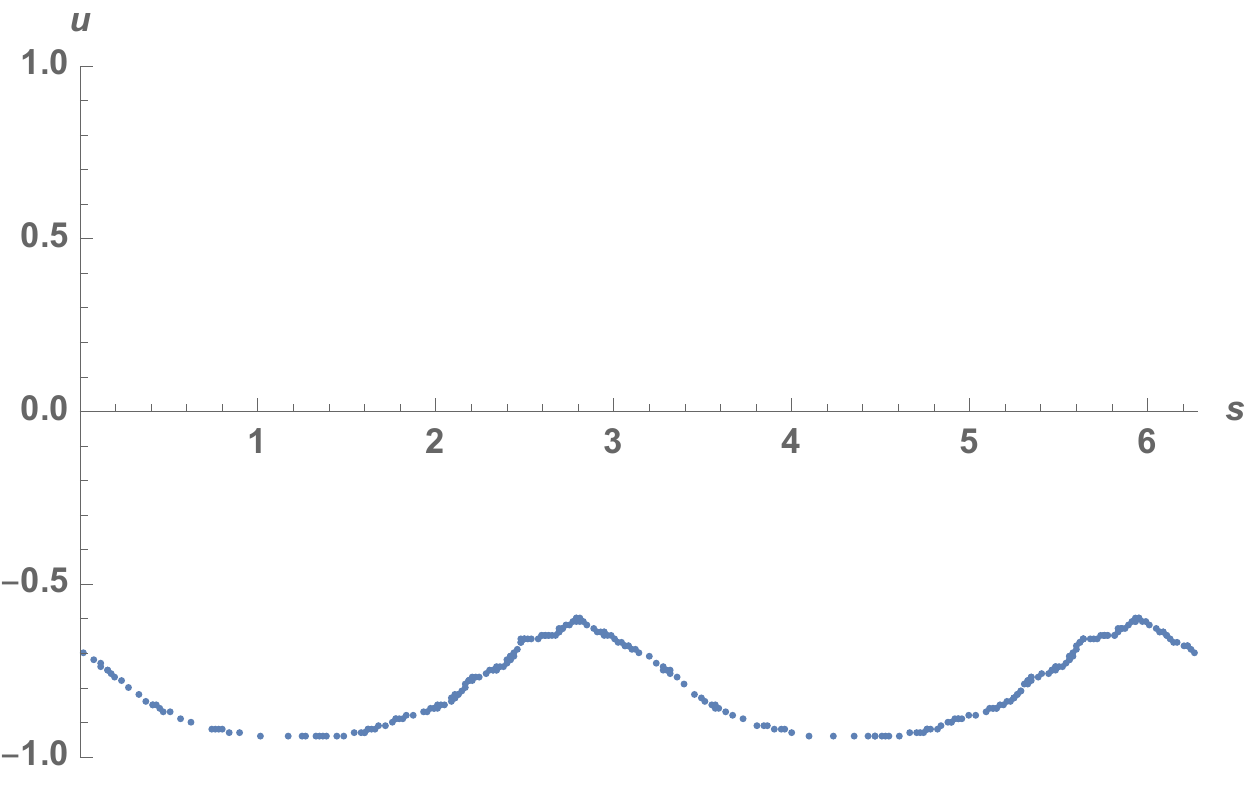} \\  
(e) \includegraphics[width =0.50\textwidth]{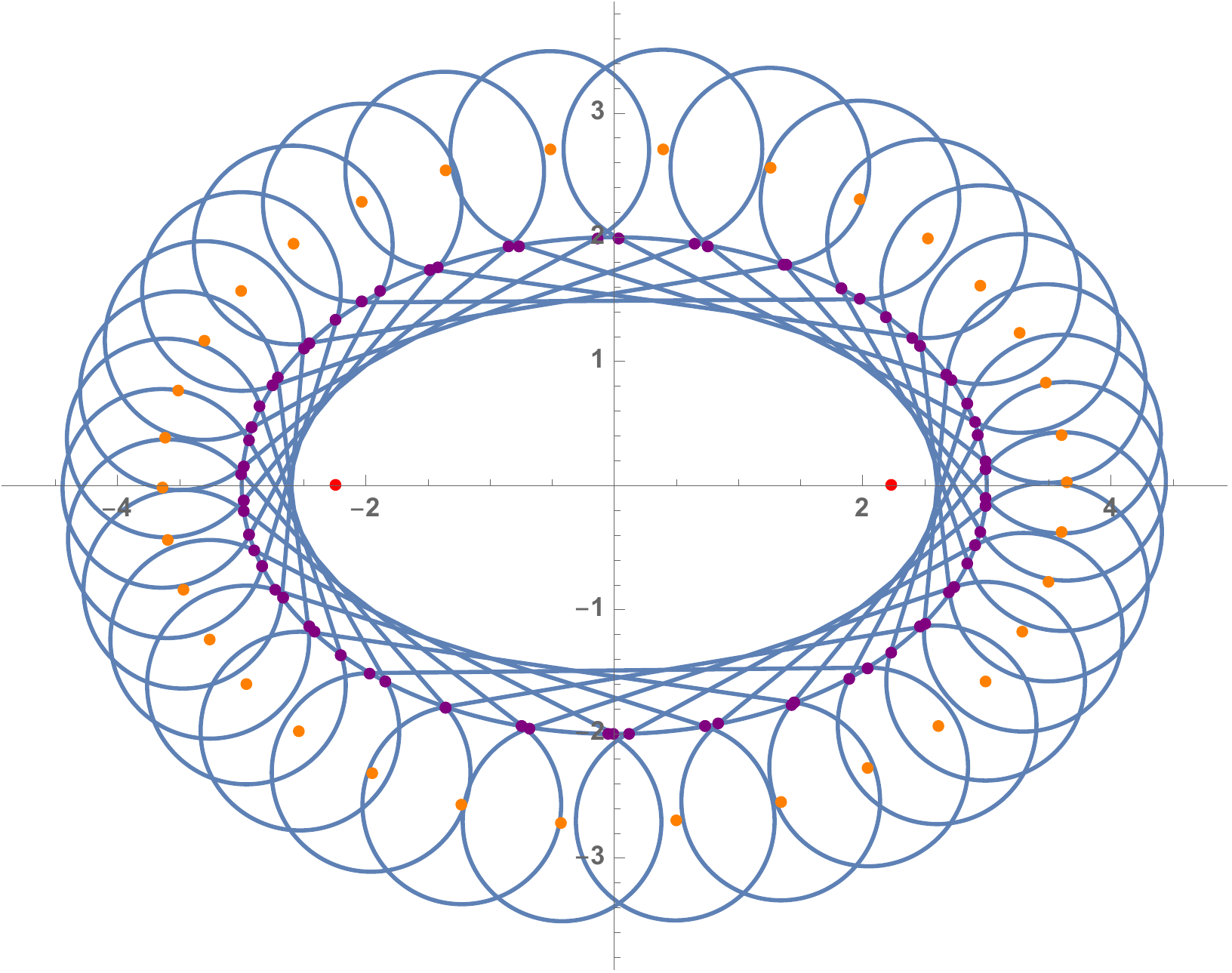} & (f) \includegraphics[width =0.45\textwidth]{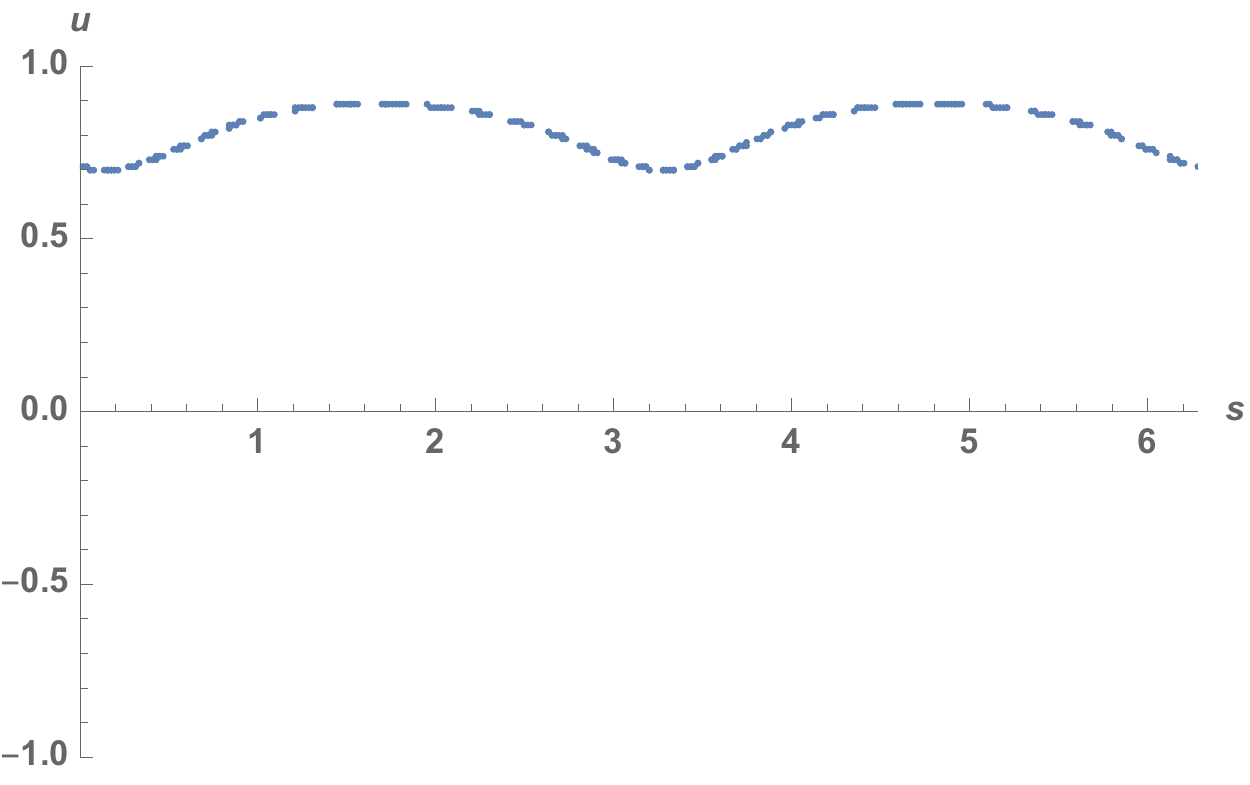} \\
\end{tabular}
\caption{Caustics in an ellipse for the three valid regimes: (a) near $u=-1$ and $\mu < \rho_{min}$; (c) near $u=-1$ and $\rho_{max}< \mu$; (e) near $u = 1$; and their accompanying invariant curves in the $(\phi, u)$-plane, (b), (d), (e), respectively. The centers of the Larmor circles, foci of the ellipse, and points $P_i$ are shown. } 
\label{caustics}
\end{figure}

\noindent Each of these three maps can be interpreted via KAM theorems (\cite{D}, pg. III-8 or \cite{Mo1} pg. 52), and to do so we want to briefly define a relevant condition. 

\begin{defn}
Let $\sigma \in \R{}$. We say $\sigma$ satisfies the \textbf{Diophantine condition} if for every $\frac{p}{q}\in \Q$, there exists $\gamma,\nu \in \R{+}$ such that $$\left| \sigma - \frac{p}{q} \right| \geq \gamma q^{-\nu}.$$
\end{defn}

\begin{theorem}
Consider the inverse magnetic billiard in a strictly convex set $\Omega$ with $C^k$ boundary, $k\geq 6$. Consider the following cases:

\begin{enumerate}
\item if $0 < \mu < \rho_{min}$, define $\zeta = \theta$, $M = L - 2\pi \mu$, and $\lambda = 1$;
\item if $\rho_{max}< \mu < \infty$, define $\zeta = \theta$, $M = 2\pi\mu - L$, $\lambda = -1$;
\item or if $0 < \mu < \infty$, define $\zeta = \pi - \theta$, $M = L + 2\pi \mu$, $\lambda = -1$. 
\end{enumerate}
Then there exists $\epsilon >0$ depending upon $\mu$ and $k$ with the following significance: if $\omega \in [0,\epsilon)$ and satisfies the Diophantine condition, 
then there is an invariant curve of the form 
\begin{align*}
s &= \xi + V(\xi) \\
\zeta &= \frac{\omega}{2\mu} + U(\xi),
\end{align*}
where $U, V \in C^1$, $V(\xi +M) = V(\xi) + L-M$, $U(\xi + M) = U(\xi)$. The induced map on this curve has the form $$\xi \mapsto \xi + \lambda \omega.$$ 
\end{theorem}

\noindent Similarly to the case of inner magnetic billiards, our theorem confirms the existence of invariant curves in three cases (see Figure \ref{caustics}):
\begin{enumerate}
\item Near $u= -1$, $\theta \approx 0$ in a strong magnetic field, $\mu < \rho_{min}$. These correspond to short billiard chords plus short magnetic arcs, keeping the particle's trajectory near the boundary. 
\item Near $u=-1$, $\theta \approx 0$ in a weak magnetic field, $\rho_{max}< \mu< \infty$. These correspond to short billiard chords followed by long magnetic arcs encompassing $\Omega$ and reentering behind the original starting point, still near the boundary. 
\item Near $u = 1$, $\theta \approx \pi$ for all values of the magnetic field. These correspond to backwards billiard chords followed by most of a magnetic arc, reentering close to the starting point and staying near the boundary. 
\end{enumerate}

This approach does not give us any indication of the behavior of the map for the intermediate curvature regime, $\rho_{min}< \mu < \rho_{max}$ near $u=-1$. While numerically we do not observe any invariant curves in this region in this case, we do not have definitive proof. This is also the case in inner magnetic billiards. Moreover, our theorem also indicates that provided we have a sufficiently smooth strictly convex boundary (at least $C^6$), inverse magnetic billiards are not ergodic.

\section{Conclusions and Next Steps}

\begin{figure}[htbp]
\begin{tabular}{ l l }  
(a) \includegraphics[width =0.50\textwidth]{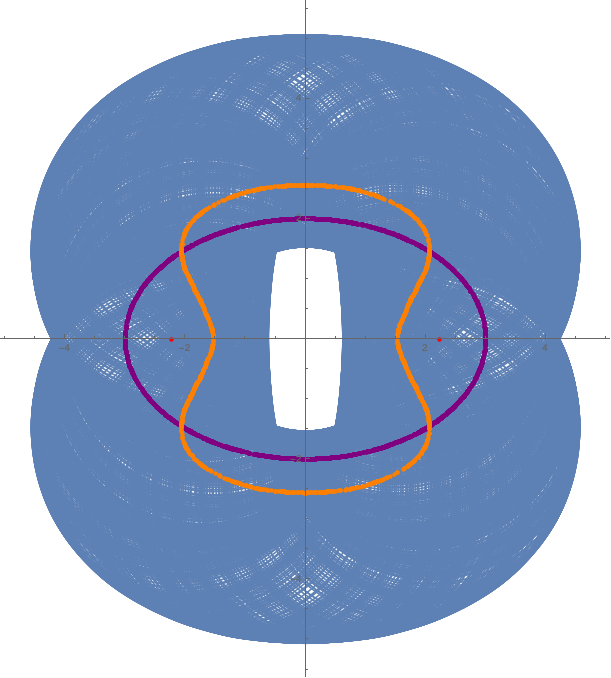} & (b) \includegraphics[width =0.45\textwidth]{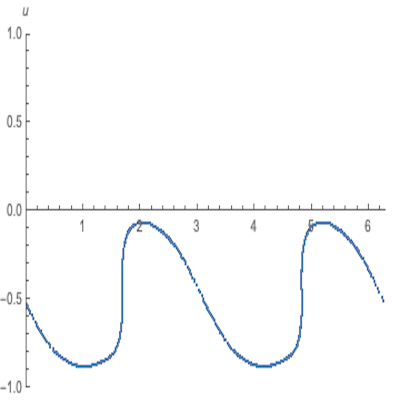} \\ 
(c) \includegraphics[width =0.50\textwidth]{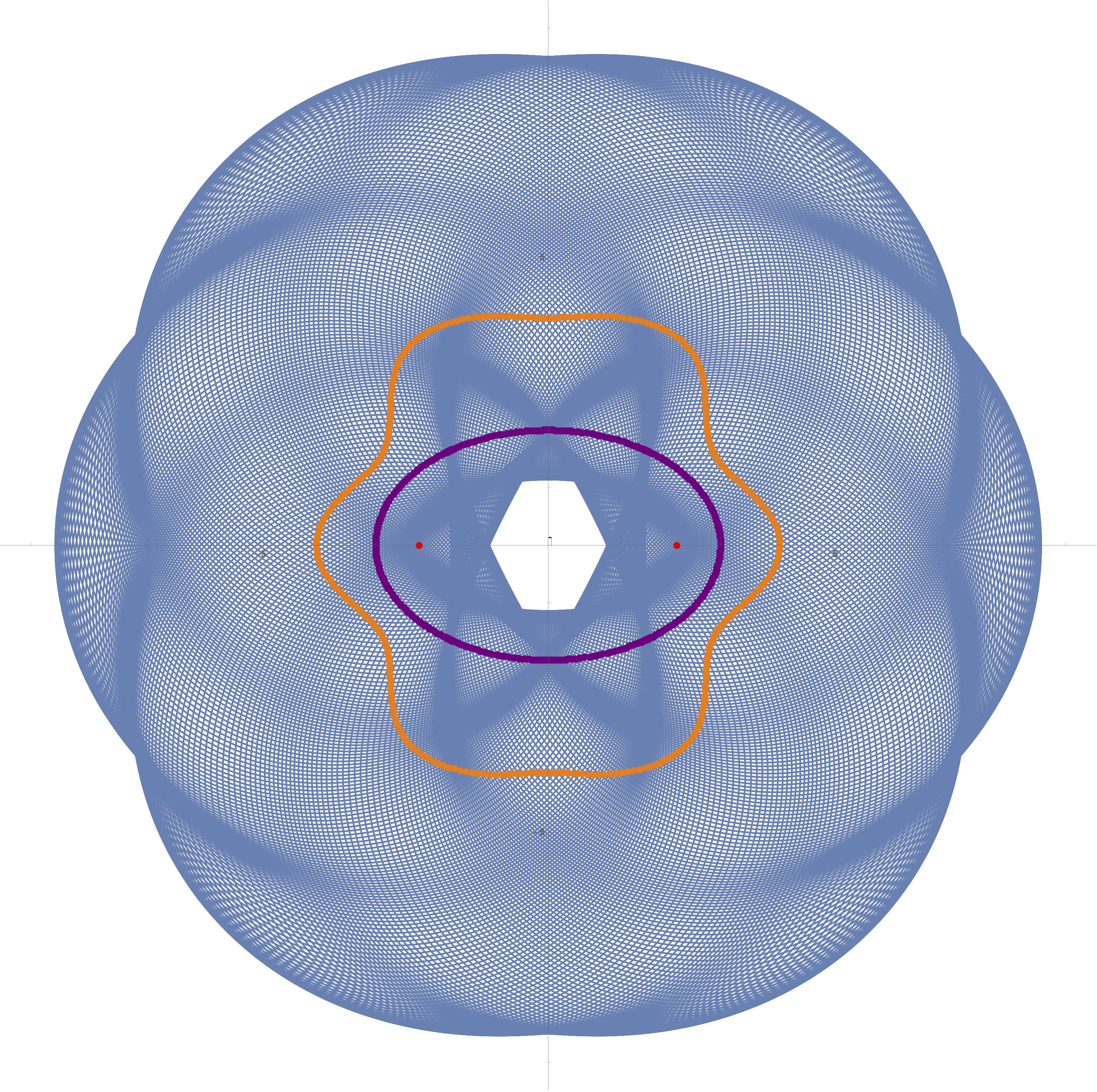} & (d) \includegraphics[width =0.45\textwidth]{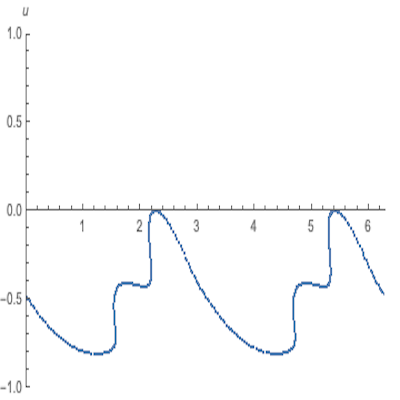} \\  
\end{tabular}
\caption{$C^0$ caustics in an ellipse for the two non-twist curvature regimes and $10^3$ iterations of $T$: (a) $\rho_{min}<\mu < \rho_{max}$; (c) $\rho_{max}< \mu$; and their accompanying invariant curves in the $(\phi, u)$-plane, (b), (d), respectively. The locus of centers of the Larmor circles appear to lie on a continuous curve.  } 
\label{nontwistcaustics}
\end{figure}

We have found that inverse magnetic billiards shares some similarities with standard and magnetic billiards while also showing concrete differences. The influence of the magnetic field on the dynamics is significant, and we have clearly seen that inverse magnetic billiards is a nontrivial perturbation of the standard billiard. 

The behavior of inverse magnetic billiards in the regimes $\rho_{min}< \mu < \rho_{max}$ and $\rho_{max}< \mu$ are not well understood at this time. For example, numerical simulations seem to show the existence of a $C^0$ caustic comprised of piecewise $C^1$ curves. Of further interest is the locus of the centers of the Larmor circles in such a case, as sometimes these centers appear to lie on a smooth simple closed curve with two axes of symmetry. See Figure \ref{nontwistcaustics}. 

 Another aspect of inverse magnetic billiards that has not been studied is the existence of \emph{outer} caustics. Figures \ref{period 9 orbits}, \ref{caustics}, \ref{nontwistcaustics} all show the existence of caustics outside of $\Omega$, and this phenomena is certainly worth investigating.


\section{Acknowledgements}
The author is grateful to Richard Montgomery, Serge Tabachnikov, and Alfonso Sorrentino for their repeated useful discussions. The author also wishes to thank N. Berglund and H. Kunz for their exposition \cite{BK}, as they provided a thorough roadmap to follow when approaching this problem. 
This material is based upon work supported by the National Science Foundation under Grant No. DMS-1440140 While the author was in residence at the Mathematical Sciences Research Institute in Berkeley, California, during the Fall 2018 semester. A portion of this work was also supported by Grant No. DP190101838 from the Australian Research Council. 


\begin{appendix}

\section{Proof of Proposition \ref{JacCalc} }
\label{jacpf}

The components of $DT_1$ are well-known (e.g. \cite{KS}, Theorem 4.2 in Part V or \cite{KT}). We provide an outline of the computation of the components of $DT_2$.

Consider a single magnetic arc, as in Figure \ref{magarcdetailsuppchi}a. Define $$\alpha_1 = \text{arg}[(X_2-X_1) + i(Y_2-Y_1)],$$  the polar angle between the positive $x$-axis and the segment $P_1P_2$. Near $P_1$ we see that $\tau_1 - \theta_1 = \alpha_1 - \chi$. A similar picture centered on $P_2$ tells us that $\tau_2 = \alpha_1 + \chi - \theta_2$. This leads to the following equations: 
\begin{align*}
\theta_1 &= \tau_1 - \alpha_1 + \chi \\
\theta_2 &= \alpha_1 + \chi - \tau_2.
\end{align*}

\noindent By construction, we have
\begin{align*}
\ell_2^2 &= (X_2 - X_1)^2 + (Y_2 - Y_1)^2 \\
\tan(\alpha_1) &= \frac{Y_2 - Y_1}{X_2 - X_1}
\label{ell2}
\end{align*} 
These equations imply
\begin{align*}
\fpd{\alpha_1}{s_1} & = \frac{1}{\ell_2^2}\left[ (Y_2 - Y_1)X^\prime(s_1) - (X_2 - X_1) Y^\prime(s_1)\right] \\
 &= \frac{1}{\ell_2^2} \left[ \ell_2\sin(\alpha_1) \cos(\tau_1) - \ell_2\cos(\alpha_1)\sin(\tau_1)\right] = \frac{\sin(\chi-\theta_1)}{\ell_2}, \\
\fpd{\ell_2}{s_1} &= \frac{1}{2\ell_2}\left[2(X_2 - X_1)(-X^\prime(s_1)) + 2(Y_2 - Y_1)(-Y^\prime(s_1)) \right] \\
&= -\left[\cos(\tau_1)\cos(\alpha_1) + \sin(\tau_1)\sin(\alpha_1)\right] = -\cos(\theta_1-\chi), \\
\fpd{\chi}{s_1} &= \frac{1}{\ell_2\cos(\chi)}\frac{\ell_2}{2\mu}\fpd{\ell_2}{s_1} 
= -\frac{\sin(\chi)\cos(\theta_1-\chi)}{\ell_2\cos(\chi)}.
\end{align*}
\noindent Next, we differentiate the angle formulas for $\theta_1$, $\theta_2$ with respect to $s_1$ to get
$$
\fpd{\theta_1}{s_1}  = \kappa_1 - \frac{\sin(2\chi-\theta_1)}{\ell_2\cos(\chi)}, \;\;\;\;\;
\fpd{\theta_2}{s_1} = - \frac{\sin(\theta_1)}{\ell_2\cos(\chi)}.$$
\noindent Repeating this process again but with respect to $s_2$ yields
$$
\fpd{\theta_1}{s_2} = \frac{\sin(\theta_2)}{\ell_2\cos(\chi)}, \;\;
\fpd{\theta_2}{s_2} = \frac{\sin(2\chi-\theta_2)}{\ell_2\cos(\chi)} - \kappa_2. $$
The above quantities determine $d\theta_1$ and $d\theta_2$ as a function of $ds_1$ and $ds_2$. Solve this linear system for $ds_2$ and $d\theta_2$ in terms of $ds_1$ and $d\theta_1$. Lastly, writing $du_i = \sin(\theta_i)d\theta_i$ we obtain the components of $DT_2$.

\section{Proof of Theorem \ref{twistcondgf}}\label{twistcondgfpf}

The proof that $T$ satisfies the twist condition $\fpd{s_2}{u_0}>0$ whenever $\mu < \rho_{min}$ is a small exercise in geometry and trigonometry. The full proof can be found in the appendices of \cite{G}.

Next, we derive our expression for the generating function. Following the proof from \cite{BK}, we take our generating function $$G(s_0,s_2) = -\ell_1 - |\gamma| + \frac{1}{\mu}\mathcal{S}$$ and break it into a magnetic field-dependent component and a magnetic field-independent component. Recall the notation used in Figure \ref{stdpic1}, and we proceed without the ``$0,2$" subscripts. Write $\mathcal{S} = Area(\mathcal{A} \cup \mathcal{S}) - \mathcal{A}$ and $Area(\mathcal{A} \cup \mathcal{S})$ to mean the area inside the circular arc $\gamma$ cut by the chord $P_1P_2$. Then 
$$G(s_0,s_2) = \left[-\ell_1 - \frac{1}{\mu}\mathcal{A}\right] + \left[-|\gamma| + \frac{1}{\mu} Area(\mathcal{A} \cup \mathcal{S})\right].$$
From elementary geometry, we see that $$\fpd{\mathcal{A}}{s_2} = \frac{1}{2}\ell_2 \sin(\chi - \theta_2).$$
\noindent Thus 
\begin{align*}
\fpd{G}{s_2} &= 0 - \frac{1}{\mu}\fpd{\mathcal{A}}{s_2} - \fpd{|\gamma|}{\chi} \fpd{\chi}{s_2} + \frac{1}{\mu}\fpd{Area(\mathcal{A} \cup \mathcal{S})}{\chi}\fpd{\chi}{s_2} \\
 &= -\sin(\chi)\sin(\chi-\theta_2) - \cos(\chi)\cos(\chi-\theta_2) = u_2. 
\end{align*}
And the calculation of the other partial derivative is simple since all factors of $G$ except for $\ell_1$ do not depend upon $s_0$. This is just the calculation from the standard billiard map, so $\fpd{G}{s_0} = -u_0.$ Hence $$G(s_0,s_2) = -\ell_1 - |\gamma| + \frac{1}{\mu}\mathcal{S}$$ is the generating function.

\section{Proof of Proposition \ref{taylorexp}}\label{taylorexppf}

\begin{proof}
We compute the coefficients in the Taylor expansion of the map $T_2$ and hence for $T$. Omitting the dependence on $s$, these terms are 

\begin{center}
\begin{tabular}{ll}
$\displaystyle \fpd{s_2}{\theta_1}(s,0) = \frac{2\mu}{1-\mu \kappa }$  
& $ \displaystyle \fpd{s_2}{\theta_1}(s,\pi) = -\frac{2\mu}{1+\mu \kappa} $ \\
$\displaystyle \fpd{\theta_2}{\theta_1}(s,0) = 1$ &  $\displaystyle \fpd{\theta_2}{\theta_1}(s,\pi) = 1$\\ 
$\displaystyle \fpd{s_2}{\theta_0}(s,0) = \frac{2}{\kappa(1-\mu \kappa)} $& $\displaystyle \fpd{s_2}{\theta_0}(s,\pi) = \frac{2}{\kappa(1+\mu \kappa)} $ \\
$\displaystyle \fpd{\theta_2}{\theta_0}(s,0) =1. $ & $\displaystyle \fpd{\theta_2}{\theta_0}(s,\pi) =1. $ 
\end{tabular}
\end{center}

First we compute $\fpd{s_2}{\theta_1}(s_1,\theta_1)$ near $\theta_1=0$. Using the approximations $\chi \approx \chi^*(\theta_1)$ from the previous appendix and applying the l'Hopital rule in the second equality, we get
\begin{align*}
L &:= \limit{\theta_1}{0^+}{\frac{\ell_2 \cos(\chi)}{\sin(\theta_2)}} = \limit{\theta_1}{0^+}{\frac{\fpd{\ell_2}{s_2}\fpd{s_2}{\theta_1}\cos(\chi) - \ell_2 \sin(\chi)\fpd{\chi}{\theta_1}}{\cos(\theta_2)\fpd{\theta_2}{\theta_1}}} \\
&= \limit{\theta_1}{0^+}{\frac{\cos(\chi - \theta_2) \fpd{s_2}{\theta_1} \cos(\chi) - \ell_2 \sin(\chi) \fpd{\chi}{\theta_1}}{\cos(\theta_2) \left[ \frac{\sin(2\chi-\theta_2)}{\sin(\theta_2)} - \kappa_2 \frac{\ell_2 \cos(\chi)}{\sin(\theta_2)}\right]} } \\
&=\frac{L}{2c-1 - \kappa_1 L}.
\end{align*}
where 
$c =\frac{\rho_1}{\rho_1 - \mu } = \frac{1}{1-\mu \kappa_1}$. This tells us that $$L = \frac{L}{2c-1 - \kappa_1 L}.$$  
It follows from the convexity of $\Gamma(s)$ and \cite{KS} (Theorem 4.3 in Part V) that $L<\infty$, so  $L=0$ or $L = \frac{2c-2}{\kappa_1}$. We wish to show that $L>0$. Consider the osculating circle $\mathcal{O}_\Gamma (s_2)$ at $\Gamma(s_2)$ with radius $\rho_2$. Then via elementary geometry, the length of the chord $\ell_2$ that is inside $\mathcal{O}_\Gamma(s_2)$ is exactly $2\rho_2\sin(\theta_2)$. Therefore $\ell_2 \geq 2\rho_2\sin(\theta_2)$, and so $$L \geq 2\cos(\chi)\rho_{min}>0.$$ This means $L = \fpd{s_2}{\theta_1}(s,0) = \frac{2c-2}{\kappa_1} = \frac{2\mu}{1-\mu \kappa_1}$.

 Next, we see that 
\begin{align*}
\fpd{\theta_2}{\theta_1}(s_1,0) &:= \limit{\theta_1}{0^+}{ \fpd{\theta_2}{\theta_1}(s,\theta_1)} 
	= \limit{\theta_1}{0^+}{\frac{\sin(2\chi-\theta_2)}{\sin(\theta_2)} - \kappa_2 \fpd{s_2}{\theta_1}} \\
&= \limit{\theta_1}{0^+}{2c-1 + O(\theta_1^2) - \kappa \left(\frac{2c-2}{\kappa}\right)} 
= 2c-1 - (2c-2) = 1. 
\end{align*}

Using these two computations, one can use expressions for $\fpd{s_2}{\theta_0}$ and $\fpd{\theta_2}{\theta_0}$ to derive the expressions near $0$ from the summary. Repeating analogous calculations near $\pi$ produces the expressions in the summary above. 

\end{proof}


\end{appendix}


\bibliographystyle{amsplain}
\nocite{*}
\bibliography{References}

\hrule

\end{document}